\documentclass{amsart}

\usepackage{amssymb}
\usepackage{amsmath}
\usepackage{graphics}
\usepackage{xcolor}
\usepackage[all]{xy}
\usepackage{enumerate}
\usepackage[normalem]{ulem}% overlines

\newtheorem{theorem}{Theorem}[section]
\newtheorem{proposition}[theorem]{Proposition}
\newtheorem{corollary}[theorem]{Corollary}
\newtheorem{lemma}[theorem]{Lemma}

\theoremstyle{definition}
\newtheorem{definition}[theorem]{Definition}
\newtheorem{example}[theorem]{Example}

\theoremstyle{remark}
\newtheorem{remark}[theorem]{Remark}

\numberwithin{equation}{section}

\newcommand{\F}{\ensuremath{\mathcal{F}}}

\newcommand{\RR}{\mathbb R}
\newcommand{\N}{\mathbb N}

\newcommand{\metric}{\ensuremath{ \mathtt{g} }}

\newcommand{\dualbeta}{\ensuremath{\vec{\beta}} }

\newcommand{\metriczermelo}{\ensuremath{ \mathtt{h}   } }

\newcommand{\Om}{{\mathcal O}}

\begin{document}

%\begin{frontmatter}

\title[On  singular Finsler foliation]{On singular Finsler foliation}

\author[M.M. Alexandrino]{Marcos M. Alexandrino}

\author[B. Alves]{Benigno O. Alves}

\author[M. A. Javaloyes]{Miguel Angel Javaloyes }

%\thanks{ }

\address{ Marcos M. Alexandrino, \,   
Instituto de Matem\'{a}tica e Estat\'{\i}stica\\
Universidade de S\~{a}o Paulo \\
 Rua do Mat\~{a}o 1010, 05508 090 S\~{a}o Paulo, Brazil}
\email{marcosmalex@yahoo.de, malex@ime.usp.br}

\address{Benigno O. Alves, \, 
Instituto de Matem\'{a}tica e Estat\'{\i}stica\\
Universidade de S\~{a}o Paulo \\
 Rua do Mat\~{a}o 1010, 05508 090 S\~{a}o Paulo, Brazil}
\email{gguialves@hotmail.com,  benigno@ime.usp.br}

%\author[M. A. Javaloyes]{Miguel Angel Javaloyes}
\address{Miguel \'Angel Javaloyes, \,
Departamento de Matem\'aticas, 
Universidad de Murcia \\
Campus de Espinardo,
30100 Espinardo, Murcia, Spain}
\email{majava@um.es}

\thanks{The first author was   supported by  Funda\c{c}\~{a}o de Amparo a Pesquisa do Estado de S\~{a}o Paulo-FAPESP 
(Tematicos:-2011/21362-2 and 2016/23746-6). 
The second author was supported by CNPq (PhD fellowship) and partially
supported by  PDSE-Capes (PhD sandwich program). The third author was partially supported  by
 the programme Young leaders in research 18942/JLI/13 by Fundaci\'on S\'eneca, Regional Agency for Science and Technology from the Region of Murcia and Spanish  MINECO/FEDER project reference
MTM2015-65430-P}
%Regional J.
%Andaluc\'{\i}a Grant P09-FQM-4496, by MICINN project MTM2009-10418 and Fundaci\'on S\'eneca project 04540/GERM/06. }

%\subjclass[2000]{Primary 53C12, Secondary 57R30}

%\titlerunning{On closed geodesics in $M/\F$}

%Riemannian foliations and closed geodesic on good orbifolds

%\date{ 2010.}


\begin{abstract}

In this paper we introduce the concept of singular Finsler foliation, which generalizes the concepts of Finsler actions, Finsler submersions and (regular) Finsler  foliations. We show that  if
 $\F$ is a singular Finsler foliation with closed leaves  on a  Randers manifold $(M,Z)$ 
with Zermelo data $(\metriczermelo,W),$ then  $\F$ is a singular Riemannian foliation on the Riemannian manifold $(M,\metriczermelo)$. 
As a direct consequence we infer that the regular leaves are equifocal submanifolds (a generalization of isoparametric submanifolds)
 when the wind $W$  is an 
 infinitesimal homothety of $\metriczermelo$ (e.,g when $W$ is killing vector field  or $M$ has constant Finsler curvature).

We also present a slice theorem  that locally relates  singular Finsler foliations on  Finsler manifolds 
with singular Finsler foliations on Minkowski spaces.

\end{abstract}

%\end{frontmatter}

\maketitle

\section{Introduction}

Let  $(M,F)$ be a  Finsler manifold; recall definition in Section \ref{section-preliminaries}. A partition $\F=\{L\}$ of $M$ into connected immersed smooth  submanifolds (the \emph{leaves}) is called \emph{a singular Finsler foliation} if it satisfies the following two conditions:
\begin{enumerate}
	\item[(a)]  $\mathcal{F}$ is a \emph{singular foliation}, i.e, for each  $v\in T_pL_p$ there exists a smooth vector field $X$ tangent to the leaves so that $X(p)=v$;  
	\item[(b)]  $\mathcal{F}$ is Finsler, i.e., if a geodesic  $\gamma:(a,b)\rightarrow M$, with $0\in (a,b)\in\RR$,  is orthogonal to the leaf  $L_{\gamma(0)}$, (i.e., $g_{\dot{\gamma}(0)}(\dot{\gamma}(0),v)=0$ for each $v\in T_{\gamma(0)}L$), then $\gamma$  is \emph{horizontal}, i.e., orthogonal to each leaf it meets. Here, $g_{\dot{\gamma}}$ denotes the \emph{fundamental tensor} associated with the Finsler metric $F$; see Section \ref{section-finsler-metrics}.
\end{enumerate}
As we will see in Lemma  \ref{lemma-equidistant},  part (b) above is equivalent to saying that the leaves are \emph{locally equidistant},  
where the distance between  the plaque $P_x$ and  the plaque $P_y$ does not need to be the same as the distance between the plaque $P_{y}$
and $P_{x}$. 

A typical example of a singular Finsler foliation is the partition of $M$ into orbits of  a Finsler action. Recall that 
an action $\mu: G\times M\to M$ on a  Finsler manifold $(M,F)$ is called a \emph{Finsler action} if 
$F\circ d\mu^{g}(\cdot)=F(\cdot)$ for every $g\in G$, i.e., if the action preserves the Finsler metric. 
For example, an isometric action $G\times M\to M$ on a Riemannian manifold  $(M,\metriczermelo)$ is also a Finsler action for the Zermelo metric with inicial data $(\metriczermelo,W)$, where $W$  is a $G$-invariant vector field, recall Lemma \ref{lemma-isometryRanders}. 
For more examples and results on Finsler actions, see  \cite{Deng-Hu} and \cite{Deng-book}.

Another class of examples is given by (regular) Finsler foliations. In particular the  foliations given by   the fibers  of  Finsler submersions;  
see \cite{Duran} and \cite[section 2.4]{Deng-Hu}.
Recall that a submersion  $\pi:(M,F^1)\rightarrow (B,F^2)$  between Finsler manifolds is a  \emph{Finsler submersion} if
for all $p\in M$ we have 
$d\pi_p(B_p(0,1))=B_{\pi(p)}(0,1),$ where $B_p(0,1)$ and $B_{\pi(p)}(0,1)$ 
are the unit balls of  $(T_pM,F^1_p)$ and $(T_{\pi(p)}B,F^2_{\pi (p)})$, respectively; see also Corollary 
\ref{equivalencedef}.

Recently Radeschi \cite{RadeschiClifford}   constructed infinitely many examples of non-homogeneous  singular Riemannian foliations on spheres using polynomial maps. More precisely he generalized the construction of Ferus,  Karcher and M\"{u}nzner
\cite {FerusKarcherMunzner} constructing Clifford foliations $\F_{C}$ on spheres
with metric $\metriczermelo_{1}$ (the round metric of constant sectional curvature equal to 1).
It is not difficult to produce 
Randers metrics        $Z$      with Zermelo data  $(\metriczermelo_{1},W)$ 
so that the foliations $\F_C$ turn to be non-homogeneous   singular Finsler foliations with respect to        $Z$,      
see  Example \ref{remark-clifford-foliation}.  More generally, as briefly explained in Remark \ref{remark-examples-singular-finsler-foliation}, 
if we start with a singular Riemannian foliation $\F$ on $(M, \metriczermelo )$ and an $\F$-\emph{foliated vector field} $W$
(i.e.,  its flow takes leaves to leaves), the foliation $\F$ turns to be a singular Finsler foliation on the Randers space $(M,       Z    )$, with Zermelo data  $(\metriczermelo, W)$.

 This leads us to the following question:   \emph{may one find a Riemannian metric $\metric$ for each  singular Finsler foliation $\F$ on $(M,F)$ so that $\F$ turns to be  a singular Riemannian foliation on $(M,\metric)$?} For regular foliations, this question has already been answered positively,  see  
\cite{Jozefowicz-Wolak,Miernowski-Mozgawa,Popescu-Popescu}. 
Very roughly speaking the idea in the regular case was to locally describe  the regular foliations by  submersions 
and to produce the smooth metric by averaging  the fundamental tensors on the  base spaces  of these submersions. In the singular
case one should be quite more careful to produce a smooth metric, because the local description of singular foliations  is not so
trivial as in the regular case.    Our first main result gives a positive
answer to the above question in the case of Randers spaces.

\begin{theorem}
\label{theorem-Randersfoliation}
Let $\F=\{L\}$ be a singular Finsler foliation with closed leaves  on a  Randers manifold $(M,Z)$ with Zermelo data $(\metriczermelo,W)$. Then  $\F$ is a singular Riemannian foliation on the Riemannian manifold $(M,\metriczermelo)$. In addition the wind $W$ is 
an $\F$-foliated vector field. 

\end{theorem}

%%%%%%%%%%%%%%%%%%% 

We say that the wind  $W$ of a Randers space  $(M,Z)$ with Zermelo data $(\metriczermelo,W)$, 
 is \emph{an infinitesimal homothety of} $\metriczermelo$ if 
 $\mathcal{L}_{W} \metriczermelo=-\sigma \metriczermelo$, where $\mathcal{L}_W$ is the Lie derivative in the direction of $W$ and $\sigma$
is a constant. In this particular, but important  class of Randers space one knows relations between geodesics of $\metriczermelo$ and the geodesics of $Z$, see e.g. \cite{Robles-geodesics-constant-curvature}. 
This already suggests that the above theorem could be useful to infer a  property (namely equifocality) whose Riemannian counterpart  has been playing a fundamental role  in the theory of Singular Riemannian Foliations. 

Given a singular foliation $\F$ on a complete Finsler space $M$, we say that a regular leaf $L$ of $\F$ is 
an \emph{equifocal submanifold}, if for each $p\in L$, there exists a 
neighborhood $U\subset L$ of $p$ so that for each basic vector field $\xi$ along $U$ the 
 \emph{(future) endpoint map} $\eta_{\xi}:U\to M$, defined as $\eta_{\xi}(x)=\exp_{x}(\xi)$, 
has derivative with constant rank. In addition $\eta_{\xi}(U)\subset L_{q}$ where $q=\eta_{\xi}(p)$. 

The concept of equifocal submanifold was introduced by Terng and Thorbergsson \cite{Terng-Thorbergsson-equifocality} to generalize the concept of isoparametric submanifolds. As proved in \cite{AlexToeben2} each regular leaf of a singular Riemannian foliation on a complete Riemannian manifold  is equifocal. 
The equifocality   has been used in the study of topological properties (see e.g. \cite{Alexandrino-isoparametricsmaps, Alex-fundamentalgroup}), metric properties (see e.g. \cite{Alex6, AlexRadeschi-isometric-flow}) 
and semi-local dynamical behavior (see e.g. \cite{AlexRadeschi-Molinos-conjecture})  of singular Riemannian foliations.   
It also plays a relevant role in the  Wilking's proof of the  smoothness of Sharafutdinov projection, i.e., the metric projection  into the soul of an  open non negative curved space; see \cite{Gromoll-Wallschap-book, Wilking=duality}.

\begin{corollary}
\label{the-corollary}
Assume that  the wind  $W$ of a  Randers space  $(M,Z)$ with Zermelo data $(\metriczermelo,W)$ 
 is an infinitesimal homothety of $\metriczermelo$. Also assume that $\metriczermelo$ and $W$ are complete. 
Let $\F$ be a singular Finsler foliation with closed leaves  on  $(M,Z)$. Then
the regular leaves of $\F$ are equifocal. 
\end{corollary}

%%%%%%%%%%%%%%%%%

In order to prove Theorem \ref{theorem-Randersfoliation}  we present in Propositions \ref{proposition-metricFS} and \ref{proposition-sliceflat-metric} a \emph{slice reduction} that  locally  relates  singular Finsler foliations on Finsler manifolds  with  singular Finsler foliations on  Minkowski spaces.  More precisely we prove the next result:

\begin{theorem}[Slice theorem]\label{thm:slice}
Let $\F=\{L\}$ be a singular Finsler foliation with closed leaves on a  Finsler manifold $(M,F)$. 
Then, given a point $q\in M$,  
there exists a slice $S_q$ transversal to $L_q$ (i.e. $T_{q}M= T_{q}S_{q}\oplus T_{q}L_{q}$) 
and a  Finsler metric $\widehat F$ on 
$S_q$  with the following properties:
\begin{enumerate}
\item[(a)] the leaves of the slice foliation $\F_{q}=\F\cap S_{q}$  endow $(S_q,\widehat F)$ of a structure of 
singular Finsler foliation;   
\item[(b)]   the distance between the leaves of $\F_{q}$ 
(with respect to $\widehat F$)  and the distance between the leaves of $\F$ (with respect to $F$) coincide  locally; 
\item[(c)] the slice foliation  $\F_{q}$ on $(S_q,\widehat F)$
 is foliated-diffeomorphic to a  singular Finsler foliation on        an open subset of       Minkowski space.
\end{enumerate}
\end{theorem}

\begin{remark}
We stress that both theorems  and the corollary are proved under a hypothesis weaker than having closed leaves, namely, the leaves are locally closed (see Definition \ref{locclosed}).
\end{remark}

This paper is organized as follows. In Section \ref{section-preliminaries} we fix some notations and briefly review a few facts about
Finsler Geometry. In section \ref{section-facts-SFF} we sum up a few results on singular Finsler foliations that are analogous to the classical results on singular Riemannian foliations. In particular we prove Propositions \ref{proposition-metricFS} and  \ref{proposition-sliceflat-metric},  which are collected above in Theorem \ref{thm:slice}.  In 
Section~\ref{section-singular-Finsler-Foliation-Randers} we prove 
Theorem \ref{theorem-Randersfoliation} and Corollary \ref{the-corollary} .

This paper establishes a connection between the theory of singular foliations and Finsler Geometry.
This leads us to the natural question  about how much the reader should know about these two topics. 
In this presentation we tried to balance brevity of explanations with the amount of prerequisites that the reader needs. In this way we hope to have made the main ideas accessible to at least  two kinds of readers, 
those who know Finsler geometry  but may not have previous experience with singular foliations 
and those with  experience in the theory of singular foliations but without previous contact with Finsler geometry.

\

\emph{Acknowledgements.} The second author thanks the Department of Mathematics of the University of Murcia
for its hospitality during his research stay at Murcia-Spain in the period of  April-September 2017. 
The first and third  authors thank the  University of Murcia and IME-USP for their  hospitality during their  visits to these institutions.

%%%%%%%%%%%%%%%%%%%%%%%%%%%%%%%%%%%%%%%%%%%%%%%%%%%%%%%%%%%%%%%%%%
%%%%%%%%%%%%%%%%%%%%%%%% PRELIMINARIES
%%%%%%%%%%%%%%%%%%%%%%%%%%%%%%%%%%%%%%%%%%%%%%%%%%%%%%%%%%%%%%%%%%%%%%%%%%%
\section{Preliminaries}
\label{section-preliminaries}
%%%%%%%%%

In this section we fix some notations and briefly review a few facts about Finsler Geometry
which will be used in this paper. For more details        see \cite{Duran,BaChSh00,Miguel-Chern-connection,Zhongmin-Shen}.

\subsection{Finsler metrics, fundamental tensor and orthogonal cone}
\label{section-finsler-metrics}

Let us begin by introducing the concepts at the level of vector spaces. Let $V$ be a vector space and $F:V\rightarrow [0,+\infty)$ a function. We say that $F$ is a {\it Minkowski norm} if
\begin{enumerate}
\item $F$ is smooth on $V \setminus  \{0 \}$,
\item $F$ is positive homogeneous of degree 1, that is $F(\lambda v)=\lambda\, F(v)$ for
every $v\in V$ and $\lambda>0$,
\item for every $v\in V \setminus \{ 0\}$, the \emph{fundamental tensor} of $F$ defined as
\begin{equation}\label{fundtensor}
 g_{v}(u,w)=\frac{1}{2}\frac{\partial^{2}}{\partial t\partial  s} F^{2}(v+tu+sw)|_{t=s=0}
 \end{equation}
for any $u,w\in V$ is a positive-definite bilinear  symmetric form. 
\end{enumerate} 
We will say that $(V,F)$ is a \emph{Minkowski space}. 

Let us consider  a manifold $M$.  
We say that a  function $F:TM\to [0,+\infty)$ is a \emph{Finsler metric} if
\begin{enumerate}
\item $F$ is smooth on $TM\setminus \bf 0$,
\item for every $p\in M$, $F_p=F|_{T_pM}$ is a Minkowski norm on $T_pM$.
\end{enumerate}
 In such a case, if $\tau:TM\rightarrow M$ is the natural projection and $v\in TM\setminus \bf 0$, we will say that the fundamental tensor $g_v$ of $F$ is the one defined in \eqref{fundtensor} for $F_{\tau(v)}$. 
Let us now sum up some properties of the fundamental tensor.
\begin{lemma} Let $(M,F)$ be a Finsler manifold and, given $v\in TM\setminus \bf 0$, let $g_{v}$ be the fundamental tensor defined in \eqref{fundtensor}. Then
\label{lemma-properties-fundamentaltensor}
\begin{enumerate}
\item[(a)] $g_{\lambda v}= g_{v}$ for $\lambda>0$,
\item[(b)] $g_{v}(v,v)=F^{2}(v)$,
\item[(c)] $g_{v}(v,u)=\frac{1}{2} \frac{\partial}{\partial z} F^{2}(v+ zu)|_{z=0}.$
\end{enumerate}
\end{lemma}
\begin{remark}
\label{remark-legrendre-transformation}
Part (c) above implies in particular that the Legrendre transformation  $\mathcal{L}:TM\setminus{\bf 0}\mapsto TM^{*}\setminus{\bf 0}$ 
associated with   $\frac{1}{2}F^{2}$ can be computed as $\mathcal{L}(v)= g_{v}(v,\cdot).$ 
Also recall that  $\mathcal{L}:TM\setminus {\bf 0} \mapsto TM^{*} \setminus \bf 0$ is a diffeomorphism. 
\end{remark}

We will also need the concept of orthogonal cone. Given a submanifold $P$ on a Finsler manifold $(M,F)$ we say that a 
vector $v\in T_{q}M$ is \emph{orthogonal  to} $P$ \emph{at} $q$   
if $g_{v}(v,u)=0$ for all $u\in T_{q}P$. The set of all (non zero) normal
vectors to $P$ at $q$, denoted as $\nu_{q}P$ is called the \emph{orthogonal cone} and, as the name suggests, it  is not always a subspace but a cone.        In particular, it is a submanifold of $T_{q}M$ of dimension $n-\dim P$. Moreover, $\nu(P)$ will denote the space of all orthogonal vectors to $P$ and $\nu^1(P)$, the space of unit orthogonal vectors.  
\begin{proposition}\label{proposition-nuP}
     Let $(M,F)$ be a Finsler manifold and $P$ a submanifold of $M$. Then $\nu(P)$ and $\nu^1(P)$ are smooth submanifolds and the restrictions of the canonical projection $\tau:\nu(P)\rightarrow P$ and $\tau:\nu^1(P)\rightarrow P$ are submersions.   
\end{proposition}
\begin{proof}
     Let $\mathcal{N}(P)$ be the submanifold of $TM^{*}\setminus \bf 0$ of the annihilators of $P$ and $\tau^*:TM^*\rightarrow M$ the canonical projection. Then $\nu(P)=\mathcal{L}^{-1}(\mathcal{N}(P))$ (recall that the Lengendre transformation is a diffeomorphism, Remark \ref{remark-legrendre-transformation}) and then it is a submanifold. Moreover, $\nu^1(P)=\nu(P)\cap T^1P$, where $T^1P$ is the $F$-unit bundle of $P$. As $\nu(P)$ and $T^1P$ are transversal, we conclude that $\nu^1(P)$ is a submanifol. Moreover,  it is straightforward to check that $\tau^*|_{\mathcal{N}(P)}:\mathcal{N}(P)\rightarrow P$ is a submersion and $\tau:\nu(P)\rightarrow P$ is the composition $\tau^*|_{\mathcal{N}(P)}\circ \mathcal L|_{\nu(P)}$, and hence, $\tau|_{\nu(P)}$ is a submersion. That $\tau:\nu^1(P)\rightarrow P$ is a submersion follows using that $\tau:\nu(P)\rightarrow P$ is a submersion and the fact that the tangent space to $\nu(P)$ can be decomposed as the sum of the tangent space to $\nu^1(P)$ and a vertical part.     
\end{proof}
\begin{proposition}\label{proposition-tnu}
     Let $(V,F)$ be a Minkowski space, $W\subset V$ a linear subspace of $V$. Then $T_v\nu_0(W)=\{ u\in V: g_v(u,w)=0\, \forall w\in W\}$, where $g$ is the fundamental tensor of $F$.   
\end{proposition}
\begin{proof}
     Recall that the Cartan tensor of $F$ can be defined as 
\[C_v(w_1,w_2,w_3)=\frac{1}{2} \frac{\partial}{\partial z} g_{v+z w_1}(w_2,w_3)\]
for $v\in V\setminus 0$ and $w_1,w_2,w_3\in V$. Given $u\in T_v\nu_0(W)$, there exists a smooth curve $\alpha:(-\epsilon,\epsilon)\rightarrow \nu_0(W)$ for some $\epsilon>0$, such that $u=\alpha'(0)$ and $v=\alpha(0)$. In particular, we have that $g_{\alpha(t)}(\alpha(t),w)=0$ for all $t\in(-\epsilon,\epsilon)$ and $w\in W$. Moreover, 
\[ \frac{d}{dt}g_{\alpha(t)}(\alpha(t),w)= 2C_{\alpha(t)}(\alpha'(t),\alpha(t),w)+g_{\alpha(t)}(\alpha'(t),w)\]
and by homogeneity, the first term on the right side is zero (last formula is a very particular case of the compatibility of the Chern connection with $g$, but as we are in a tangent space, Chern connection can be indentified with a standard derivative). Therefore, it follows that $g_{v}(u,w)=0$ for all $w\in W$. This implies the inclusion $``\subseteq"$ and as the dimensions are the same, they coincide.   
\end{proof}
\subsection{Geodesics and the exponential map} Given a Finsler manifold $(M,F)$, geodesics can be defined as critical points of the energy functional 
\begin{equation}
E_F(\gamma)=\frac 12\int_a^b F(\dot\gamma)^2 ds
\end{equation}
defined for piecewise smooth curves $\gamma:[a,b]\rightarrow M$ with fixed endpoints. Alternatively, they can also be defined with one of the connections associated with Finsler metrics, \cite[\S 2.4]{BaChSh00}. In particular, observe that the Chern connection can be interpreted as a family of affine connections (see for example \cite{Miguel-Chern-connection,JavSoa15}). As a matter of fact, for every vector $v\in TM$, there exists a unique maximal geodesic $\gamma_v:(a,b)\rightarrow M$ such that $\dot\gamma_v(0)=v$ and one can define the exponential map in an open subset $\mathcal U\subset TM$ for those vectors $v$ such that the maximal interval of definition $(a,b)$ of $\gamma_v$ includes the value $1$. Then $\exp:\mathcal U\rightarrow M$ is defined as $\exp(v)=\gamma_v(1)$ and it is smooth away from the zero section and $C^1$ on the zero section \cite[\S 5.3]{BaChSh00}. In particular, this allows one to show that geodesics minimize the distance associated with Finsler metrics in some interval. More generally, geodesics minimize (in some interval) the distance with any submanifold        $P$ whenever they are orthogonal to $P$      \cite{NotaViz}.

%%%%%%%%%%%%%%%%%%%%%%%%%%%%%%%%%%%%%%%%%%%%%%%%%%%%

\subsection{Zermelo and Randers metric}
Let us now recall a central example of Finsler metrics that appears naturally in several applications; see e.g., \cite{Bao-Robles-Shen, CapJavSan10}. These metrics can be described in several ways. Classically, a Randers metric in a manifold $M$, $R:TM\rightarrow [0,+\infty)$, is defined as $R(v)=\sqrt{a(v,v)}+\beta(v)$, where $a$ is a Riemannian metric in $M$ and $\beta$, a one-form such that $\|\beta\|_a<1$ at every point. Here the norm $\|\cdot\|_a$ is computed using the norm of the Riemannian metric $a$. On the other hand,
a Finsler metric  $Z:TM\to [0,+\infty)$ is said to be a \emph{Zermelo metric}
with \emph{Zermelo Data} $(\metriczermelo,W),$  for a Riemannian metric $\metriczermelo$ and  smooth vector field  $W$ with $\metriczermelo(W,W)<1$  on the whole $M$  (the \emph{wind}), if $Z$ is the solution of 

\begin{equation*}
	\metriczermelo(\frac{v}{Z(v)}-W,\frac{v}{Z(v)}-W)=1, v \in TM\setminus \bf 0.
\end{equation*}
Denote $\mathcal{I}^{Z}_p$ the \emph{indicatrix} of $Z$ at $p\in M$, i.e. $\mathcal{I}^{Z}_{p}=\{v\in T_{p}M; Z(v)=1\}$. 
The above equation is equivalent to saying that 
 $\mathcal{I}_{p}^{Z}=\mathcal{I}_{p}^{\metriczermelo}+W.$ It turns out that both types of metrics, Randers and Zermelo, are different ways of describing the same family
 (see \cite[\S 1.3]{Bao-Robles-Shen} and also \cite[Prop. 3.1]{BiJa11}).

\begin{lemma}
\label{lemma-Randes-equivalence}
Let $Z:TM\to [0,+\infty)$ be a Finsler metric in a manifold $M$. Then the following claims are equivalent:
\begin{enumerate}
 \item[(a)] $Z$  is a Zermelo metric with Zermelo data $(\metriczermelo,W)$ of a Riemannian metric $\metriczermelo$ and a wind $W$  with $\metriczermelo(W,W)<1$ at every point; 
\item[(b)] $Z$ is a Randers metric,  $Z(v)=\sqrt{a(v,v)}+\beta(v)$, for a Riemannian metric $a$ and a 1-form $\beta$ with $\|\beta\|_a<1$ at every point.
\end{enumerate}
In addition, given a Zermelo data $(\metriczermelo,W)$, the pair $(a,\beta)$   is obtained as follows:
	$$a(u,w)=\frac{\lambda \metriczermelo(u,w)+\metriczermelo(u,W)\metriczermelo(w,W)}{\lambda^2} \mbox{ \ ,  \ }	\beta(w)=-\frac{\metriczermelo(w,W)}{\lambda},$$
where $\lambda:=1-\metriczermelo(W,W)$. Moreover,  if $Z$ is a Randers metric as in $\rm (b)$, its Zermelo data is given by
$$\metriczermelo(w,u)=\mu (a(w,u)-a(\dualbeta,w)a(\dualbeta,u)), \mathrm{\, and \, }  W=-\frac{\dualbeta}{\mu},$$ 
where $\beta(\cdot)=a(\cdot,\dualbeta)$ and $\mu:=1-a(\dualbeta,\dualbeta)$.
\end{lemma}

Once we have recalled the equivalence between the definitions of Randers and Zermelo metrics, we  can present the explicit calculation of the
fundamental tensor \cite[Cor. 4.17]{JavSan14}.

\begin{lemma}
\label{lemma-product-gvvu}
Consider a Randers metric  $Z=\alpha +\beta=\sqrt{a}+\beta$ with Zermelo data $(\metriczermelo, W)$. 
Then its fundamental tensor is given by: 
 \begin{equation*}
 g_v(w,u)=\frac{Z(v)}{\alpha(v)}\big[a(w,u)-\frac{a(v,w)a(v,u)}{\alpha^2(v)}\big]+(\frac{a(v,w)}{\alpha(v)}+\beta(w))(\frac{a(v,u)}{\alpha(v)}+\beta(u)), 
 \end{equation*}
 for $v, u, w\in T_pM$ and $v\neq 0$. In particular	
\begin{equation*}
 g_v(v,u)=Z(v)\Big(\frac{a(u,v)}{\alpha(v)}+\beta(u)\Big)=\frac{Z(v)}{\mu \alpha(v)}\Big( \metriczermelo (v-Z(v)W,u )  \Big).	\end{equation*}  
\end{lemma} 
We stress that the calculation of second equation above, follows easily  from   Lemma \ref{lemma-properties-fundamentaltensor}. Furthermore,

%\begin{lemma}
%Consider a Randers metric  $Z=\alpha +\beta=\sqrt{a}+\beta$ with Zermelo data $(\metriczermelo, W)$. Then 
%$$ g_v(v,u)=\frac{Z(v)}{\lambda \alpha(v)}\Big( \metriczermelo (v-Z(v)W,u )  \Big) $$
%\end{lemma}

\begin{lemma}
\label{lemma-isometryRanders}
Let $\psi: (M_{1},Z_{1})\to (M_{2},Z_{2})$ be a diffeomorphism between two Randers spaces with $Z_i=\sqrt{a_i}+\beta_i$, with Zermelo data $(h_i,W_i)$, $i=1,2$. Then the following three statements
are equivalent:
\begin{enumerate}
\item[(a)]  $\psi$ lifts to an isometry between $Z_1$ and $Z_2$,
\item[(b)]  $\psi^{*}a_{2}=a_1$ and $\psi^{*}\beta_{2}=\beta_{1}$,
\item[(c)]  $\psi^{*}h_{2}=h_{1}$ and $\psi_{*}W_{1}=W_{2}$.
\end{enumerate}
\end{lemma}

Finally, given a vector space $V$ and a Minkowski norm $Z$ of Randers type on it, we will say that   $(V,Z)$ is a \emph{Randers-Minkowski} space.

%%%%%%%%%%%%%%%%%%%%%%%%%%%%%%%%%%%%
\subsection{Finsler Submersions} As recalled in the introduction, a submersion between Finsler manifolds $\pi:(M,F_1)\rightarrow (B,F_2)$   is a  \emph{Finsler submersion} if 
$d\pi_p(B_p(0,1))=B_{\pi(p)}(0,1),$ for every $p\in M$,  where $B_p(0,1)$ and $B_{\pi(p)}(0,1)$ 
are the unit balls of the Minkowski spaces $(T_pM,(F_1)_p)$ and $(T_{\pi(p)}B,(F_2)_{\pi (p)})$, respectively. The next characterization was given in \cite[Prop. 2.1]{Duran}.

\begin{lemma}
If $\pi:(M,F_1)\rightarrow (B,F_2)$ is a Finsler submersion, then for every $w\in T_{\pi(p)}B$ we have
$F_{2}(w)=\inf \{ F_{1}(v);  v\in T_{p}M \ \mathrm{ and} \ d\pi_{p}(v)=w \}$. 
\end{lemma}
It is clear that $F_{2}(d\pi(v))\leq F_{1}(v)$. The vectors for which the equality holds are said to be
\emph{horizontal} and form the \emph{horizontal cone} of $\pi$. 
The above lemma, the convexity of $F^{2}$ and  Lemma \ref{lemma-properties-fundamentaltensor} imply that 
to be horizontal with respect to the Finsler submersion $\pi$ or to be orthogonal to the fibers of $\pi$  (in the sense of Section \ref{section-finsler-metrics}) are equivalent concepts.  Therefore
 the horizontal cones of $\pi$  are equal  to the orthogonal cones to the fibers of $\pi$.

\begin{lemma}\label{lem:exMin}
Let $\pi:V_1\to V_2$ be a linear submersion, with $V_1,V_2$ vector spaces. If $F_1$ is a Minkowski norm on $V_1$ and $\nu(\pi)_0$  is the subset of  the $F_1$-orthogonal vectors to $\pi^{-1}(0)$, then
\begin{itemize}
\item[(a)] the map $\varphi:=\pi|_{\nu(\pi)_0}:\nu(\pi)_0\to V_2\setminus \{0\} $ is a diffeomorphism which can be extended continuously to zero and it is positive homogeneous,
\item[(b)] the function $F_2(v)=F_1(\varphi^{-1}(v))$ is the unique Minkowski norm on $V_2$ that makes $\pi:(V_1,F_1)\to (V_2,F_2)$ a linear Finsler submersion.
\end{itemize}
\end{lemma}
\begin{proof}
For part $\rm (a)$, observe that 
\begin{equation}\label{normalvectors}
v\in \nu(\pi)_0 \text{ if and only if $\pi^{-1}(\pi(v))$ is tangent to $\mathcal{I}_1(F(v))$}
\end{equation}
where $\mathcal{I}_1(r)=\{v\in V_1: F_1(v)=r\}$. This is a consequence of  part $(c)$ of Lemma~ \ref{lemma-properties-fundamentaltensor}.
Given $w\in V_2\setminus \{0\}$, let $r=\inf \{ F_1(v): v\in \pi^{-1}(w)\}$. Then by the strong convexity of $\mathcal{I}_1(r)$, there exists a unique $v_r\in V_1$ such that $\pi(v_r)=w$ and $F_1(v_r)=r$, which turns out to be the unique element in $\nu(\pi)_0$ that projects into $w$, as $\pi^{-1}(w)$ is transverse to every $\mathcal{I}_1(r')$ with $r'>r$. This proves that $\varphi$ is one-to-one. As it is the restriction of a smooth map to a smooth submanifold      (see Proposition \ref{proposition-nuP}),    it is smooth. Moreover,      observe that, using Proposition \ref{proposition-tnu}, it follows that the tangent space to $\nu(\pi)_0$ at $v$ is given by the $g_v$-orthogonal space to $\pi^{-1}(w)$,      %(proceed as in the proof of \cite[Lemma 3.3]{JavSoa15}) 
and then  $d_v\varphi(u)=0$ for $u\in T_{v}\nu(\pi)_0$ if and only if $u=0$.
Applying the inverse function theorem,  it follows that $\varphi$ is a diffeomorphism. For part $\rm (b)$, it is enough to see that $F_2$ is strongly convex, which is equivalent to prove that the second fundamental form of its indicatrix is positive definite with respect to the opposite to the position vector \cite[Theorem 2.14 (iv)]{JavSan14}. Given an arbitrary curve in $\mathcal{I}_2$, $\rho:[0,1]\rightarrow  \mathcal{I}_2$, where $\mathcal{I}_2$ is the indicatrix of $F_2$, we have that $\tilde\rho=\varphi^{-1}\circ \rho$ is a curve with image in $\mathcal{I}_1$ and 
\begin{eqnarray*}
(\pi \circ \tilde\rho)''(0)&=&\pi(\tilde\rho''(0))=\pi(\nabla^0_{\tilde\rho'(0)}\tilde\rho'(0)) -\tilde\sigma^{\tilde\xi}(\tilde\rho'(0),\tilde\rho'(0))\rho(0),
\end{eqnarray*}
where $\nabla^0$ is the induced connection on $\mathcal{I}_1$ by the affine connection of $V_1$ and $\tilde\sigma^{\tilde\xi}$ is the second fundamental form of $\mathcal{I}_1$ with respect to the opposite to the position vector $\tilde\xi$.      Observe that \eqref{normalvectors} implies that the tangent space to $\mathcal{I}_1$ in a horizontal vector projects into the tangent space to $\mathcal{I}_2$, and then    $\pi(\nabla^0_{\tilde\rho'(0)}\tilde\rho'(0))$ is tangent to $\mathcal{I}_2$. Therefore we conclude that the second fundamental form of $\mathcal{I}_2$ with respect to the opposite of the position vector $\xi=-\rho(0)$ is given by $\sigma^\xi(\rho'(0),\rho'(0))=\tilde\sigma^{\tilde \xi}(\tilde\rho'(0),\tilde\rho'(0))$ and then $\sigma^\xi$ is positive definite.

%%%%%%%%%%%%%%%%%%%%%%%%%

%% Prueba alternativa

%Alternatively to check that $F_2$ is strong convex it suffices to check 
%\begin{equation}
%\label{eq-strong-convex-alternatively}
%\frac{\partial^{2}}{\partial t\partial s} F_{2}^{2}(v_{2}+(s+t)w_{2})\Big{|}_{s=t=0}= 
%\frac{\partial^{2}}{\partial t\partial s} F_{1}^{2}(v_{1}+(s+t)w_{1})\Big{|}_{s=t=0}
%\end{equation}
%for all $v_{2}, w_{2}\in V_{2}$; where $v_1, w_1$ are the vectors tangent to $T_{v_{1}}\nu(\pi)$ that projects to $v_2, w_2$. 
%Equation (\ref{eq-strong-convex-alternatively}) can be checked through direct calculations, using the fact that the
%$F_{1}$-gradient of $F_{1}^{2}$ is tangent to 
%$T_{v_{1}}\nu(\pi)$. In fact using $g_{v_{1}}$ and identifying $T_{v_{1}}\nu(\pi)$ with 
%$\mathbb{R}^{n}\times\{0\}\subset \mathbb{R}^{n}\oplus \mathbb{R}^{m}$ and seting $f(x)=F_{1}^{2}(x+v_{1})$
%Equation (\ref{eq-strong-convex-alternatively}) follows from the next (Euclidean) claim:
%\emph{
%Let $f:\mathbb{R}^{n}\oplus \mathbb{R}^{m} \mapsto \mathbb{R}$ be a smooth function
%so that $\nabla f(0,0)$ is tangent to $\mathbb{R}^{n}\times\{0\}$
%and $h:\mathbb{R}^{n}\mapsto \mathbb{R}^{m}$ a smooth map
%so that $h(0)=0$ and $D h(0)=0$. Then }
%$$ 
% \frac{\partial^{2}}{\partial t\partial s} f((s+t)w_{1},h((s+t)w_{1}))\Big{|}_{s=t=0}=
%\frac{\partial^{2}}{\partial t\partial s} f((s+t)w_{1},0)\Big{|}_{s=t=0}
%$$

%%%%%%%%%%%%%%%%%%%%%%%%%%%%%

\end{proof}

%%%%%

\begin{lemma}\label{lem:exMin2}

Let $\pi:M\to B$ be a submersion, 
$F_1$ a Finsler metric on $M$ and $S$ a submanifold of $M$ transversal to the fibers of 
$\pi$ and with the same dimension of $B$ in such a way that 
$\pi|_S:S\to B$ is a diffeomorphism. 
Then

\begin{itemize}
\item[(a)] the subset $\nu(\pi)$ of non zero $F_1$-orthogonal vectors to the fibers of $\pi$ 
is a smooth submanifold of dimension $\dim M+\dim B$  and the restriction of the natural projection $\tau|_{\nu(\pi)}:\nu(\pi)\rightarrow M$ is a submersion, 

\item [(b)] considering at each $x\in S$ the Minkowski norm $(F_2)_{\pi(x)}$ on $T_{\pi(x)}B$ 
obtained in part (b)  of Lemma \ref{lem:exMin} for $d\pi_x:(T_xM,(F_1)_x)\to T_{\pi(x)}B$, 
we get a Finsler metric $\hat F$  in $S$  
given at each $x\in S$ by  $\hat F_x=(F_2)_{\pi(x)}\circ d\pi|_{ TS}$.

\end{itemize}

\end{lemma}

\begin{proof}
Let $\mathcal{N}$ be the submanifold of $TM^{*}\setminus \bf 0$ of the annihilators of the 
fibers of $\pi$; see \cite[Proposition 5.2]{Duran}. Then
 $\nu(\pi)=\mathcal{L}^{-1}(\mathcal{N})$, where
$\mathcal{L}(v)=g_{v}(v,\dot)$ is the Legrendre transformation $\mathcal{L}:TM\setminus {\bf 0}\to TM^{*}\setminus {\bf 0}$  associated with $\frac{1}{2}F_{1}^{2}$ and therefore it is a submanifold of $TM$ of dimension $\dim M+\dim B$ (the same as the annihilator); see
\cite[Proposition 2.4]{Duran} and Remark \ref{remark-legrendre-transformation}. Moreover, as the Legendre transformation is a global diffeomorphism  that preserves the fibers and $\tau^*|_{\mathcal N}:{\mathcal N}\rightarrow M$ is a submersion, being $\tau^*:TM^*\rightarrow M$ the canonical projection, it follows that $\tau|_{\nu(\pi)}:\nu(\pi)\rightarrow M$ is also a submersion.

In order to prove part (b), it is enough to prove that $\hat F$ is smooth. 
Let $\nu(\pi)_S=\{v\in\nu(\pi): \tau(v)\in S\}$.   
 Note that $\nu(\pi)_S=\mathcal{L}^{-1}(\mathcal{N}\cap (\tau^*)^{-1}(S))$. Therefore 
$\nu(\pi)_S$ is a submanifold and in fact  a subbundle of $TM_S=\tau^{-1}(S)$. 
As in part $(a)$, $\tau|_{\nu(\pi)_S}:\nu(\pi)_S\rightarrow S$ is a submersion. This implies that $\dim \nu(\pi)_S=2\dim(B)$ and then the only vertical vectors which are tangent to $\nu(\pi)_S$ are those tangent to $\nu(\pi)_S\cap T_pM$ for some $p\in M$. Now observe that from   Lemma \ref{lem:exMin}, it follows that $\varphi_S:=d\pi|_{\nu(\pi)_S}: \nu(\pi)_S\to TB\setminus \bf 0$  is one-to-one and we only have to prove that it is in fact a diffeomorphism, or equivalently, that $d\varphi_{S}(u)=0$ if only if $u=0$. Observe that if $d\varphi_{S}(u)=0$, then $u$ is vertical, but by the above observation, it is tangent   to $\nu(\pi)_S\cap T_pM$ for some $p\in M$. Lemma \ref{lem:exMin} implies that $u=0$ as required.

\end{proof}

We end this section with some remarks about Randers submersions. For the next results let us denote by $F_{W}$ the Finsler metric whose indicatrix is the translation of the indicatrix of a Finsler metric $F$ with a  vector field $W$  such that $F(-W)<1$ on the whole manifold. 

\begin{lemma}
\label{lemma-translation}
Let $\pi:(M,\tilde F)\rightarrow (B,F)$ be a Finsler submersion, $W$ a vector field in $B$ and $\widetilde{W}$ a  vector field in $M$ that projects onto $W$, i.e., 
$d\pi(\widetilde{W})=W$. Then $\pi:(M,{\tilde F}_{\widetilde{W}})\rightarrow (B,F_W)$ is also a Finsler  submersion.
\end{lemma}
\begin{proof}
As $\pi:(M,\tilde F)\rightarrow (B,F)$ is a Finsler  submersion,  
\begin{equation}\label{hypo}
d\pi_p(B_{{\tilde F}_p}(0,1))=B_{F_{\pi(p)}}(0,1).
\end{equation}
 Observe that $B_{({\tilde  F}_{\widetilde{W}})_p}(0,1)=B_{{\tilde F}_p}(0,1)+\widetilde{W}$ and $B_{(F_W)_{\pi(p)}}(0,1)=B_{F_{\pi(p)}}(0,1)+W$. By the linearity of $d\pi$, equation \eqref{hypo} and the relation $d\pi(\widetilde{W})=W$, we conclude that $d\pi(B_{({\tilde F}_{\widetilde{W}})_p}(0,1))=B_{(F_W)_{\pi(p)}}(0,1)$ as required.
\end{proof}

We will say that the Finsler submersion $\pi:(M,{\tilde F}_{\widetilde{W}})\rightarrow (B,F_W)$ is \emph{a translation of} $\pi:(M,\tilde F)\rightarrow (B,F)$.

\begin{proposition}
\label{proposition-translation-of-submersion}
Let $\pi:(M,Z_{1})\rightarrow (B,Z_{2})$ be a submersion between Finsler manifolds with $Z_1$ a Randers metric. Then $\pi$ is Finsler if and only if  $Z_2$ is a Randers metric and $\pi$  is the translation of a Riemannian submersion.
\end{proposition}
\begin{proof}
Let us first show that if $\pi:(M,Z_1)\rightarrow (B,Z_2)$ is Finsler, then $Z_2$ is Randers.  
Let $(\metriczermelo_1,W_1)$ be  the Zermelo data of $Z_1$.  Given $p\in M$, let $Q$ be the $\metriczermelo_1$-orthogonal subspace to the fibers in $T_pM$. Then $\psi=d\pi|_Q:Q\to T_{\pi(p)}B$ is a linear isomorphism. 
Consider on $T_{\pi(p)}B$ the metric  $\metriczermelo_2=(\psi^{-1})^{*}(\metriczermelo_1|_{Q\times Q})$. Then
$d\pi_{p}:(T_{p}M,\, \metriczermelo_{1}) \to (T_{\pi(p)}B,\metriczermelo_2)$ is a linear Riemannian submersion. From
Lemma \ref{lemma-translation}, $d\pi_{p}:(T_{p}M,\, Z_{1}) \to (T_{\pi(p)}B,\widetilde{Z}_2)$ is a linear Finsler submersion,
where $\widetilde{Z}_{2}$ is the Randers metric with Zermelo data $(\metriczermelo_2,W_2)$, for $W_2=d\pi_{p}(W_1)$. Now the unicity of  part (b) of Lemma  \ref{lem:exMin} implies
that $\widetilde{Z}_{2}$ is $(Z_{2})_{\pi(p)}$, and hence $Z_2$ is the Randers metric with Zermelo data $(\metriczermelo_2,W_2)$. Therefore $W_1$ is projectable. Applying Lemma \ref{lemma-translation} with  $-W_1$ and $-W_2$,  it follows that $\pi:(M,\metriczermelo_1)\to (B,\metriczermelo_2)$ is a Riemannian submersion as required.
 The converse follows immediately from Lemma \ref{lemma-translation}.

\end{proof}

%%%%%%%%%%%%%%%%%%%%%%%%%%%%%%%%%%%%%%%%%%%%%%%%%%%%%%%%%%%%%%%%%%%%%%%%%%%%%%%%%%
%%%%%%%%%%%%%%%%%%%%%%%%%%% NONHOMOGENEOUS EXAMPLES 
%%%%%%%%%%%%%%%%%%%%%%%%%%%%%%%%%%%%%%%%%%%%%%%%%%%%%%%%%%%%%%%%%%%%%%%%%%%%%%%%%%%%%%%%%%%%%%%%%%

\begin{example}
\label{remark-clifford-foliation}
\emph{Construction of        (non-Riemannian) singular Finsler foliations      which are non-homogeneous:} in \cite{RadeschiClifford} Radeschi constructed a polynomial map $\pi_{C}: \mathbb{S}^{2l-1}\to \mathbb{R}^{m+1}$
so that, for  $l> m+1 $ and $m\neq 1,2, 4$, the 
 preimage of $\pi_{C}$ are (non-homogeneous) leaves of  a singular Riemannian
foliation on $(\mathbb{S}^{2l-1},\F_{C})$ 
whose leaf space is  the disk $\mathbb{D}_{C}$ (i.e., the image of $\pi_{C}$).
Consider a small smooth radial   vector field $W_d$ which is zero at the center and near  $\partial \mathbb{D}_{C}$. Let $W$ be a  basic  vector field on $\mathbb{S}^{2l-1}$ that projects to $W_d$.
Set        $Z$      the Zermelo metric on $\mathbb{S}^{2l-1}$ with Zermelo data $(\metriczermelo_{1}, W)$, where 
 $\metriczermelo_{1}$ is the round metric
of constant sectional curvature 1. Then it follows from
Proposition \ref{proposition-translation-of-submersion}  that 
 $\F_{C}$ turns        out      to be a singular Finsler foliation on    
$(\mathbb{S}^{2l-1},       Z     )$, which, by Lemma \ref{lemma-isometryRanders}, has  non-homogeneous leaves. 
\end{example}

\begin{remark}
\label{remark-examples-singular-finsler-foliation}
As remarked in Proposition \ref{proposition-foliation-submersion}, if $\F$ is a singular Riemannian foliation  on a Riemannian manifold $(M,g)$, and if there exists a Finsler metric $F$ such that
$\F$ restricted to each stratum is a (regular) Finsler foliation, then $\F$ is a singular Finsler foliation on $(M,F)$. This gives us another way to contruct examples. In particular if we start with a singular Riemannian foliation $\F$ on $(M, \metriczermelo )$ and an $\F$-foliated vector field $W$, then  $\F$ 
turns to be a singular Finsler foliation on the Randers space $(M,       Z     )$, with Zermelo data  $(\metriczermelo, W)$. 
Also Proposition \ref{proposition-foliation-submersion} 
 and suspension of homomorphism allow us to construct examples of singular Finsler foliations with non closed leaves; see \cite[Chapter 5]{AlexBettiol}.
\end{remark}

%%%%%%%%%%%%%%%%%%%%%%%%%%%%%%%%%%%%%%%%%%%%%%%%%%%%%%%%%%%%%%%%%%%%%%%%%
%%%%%%%%%%%%%%%%% SOME PROPERTIES OF SINGULAR FINSLER FOLIATION
%%%%%%%%%%%%%%%%%%%%%%%%%%%%%%%%%%%%%%%%%%%%%%%%%%%%%%%%%%%%%%%%%%%%%%%%%%

\section{Some properties of Singular Finsler foliations}
\label{section-facts-SFF}

In this section we sum up a few properties about  singular Finsler foliations. 
Their Riemannian counterparts can be found in \cite{AlexRafaelDirk} and \cite{Molino}.

\begin{definition}
Let $(M,\F=\{L\})$ be a singular foliation and $F$ a Finsler metric on $M$. 
We will say that an open subset $P_q$ of a leaf $L_q$ is a \emph{plaque} if it is relatively compact  and connected.  In this case, it is possible to consider a tubular neighborhood $U$ of $P_q$,
%%%%%%%%%%%%%%%%%%%%%%%%%% 
 %%\cite{AAJ17}, BENIGNO citation
%%%%%%%%%%%%%%%%%%%%%%%%%
which we will assume to be relatively compact \cite{NotaViz}.
Recall that $U$ is called a \emph{tubular neighborhood} (of radius $\epsilon$)
if $\exp$ sends 
$\nu(P_q)\cap F^{-1}((0,\epsilon))$
diffeomorphically to  $U\setminus P_{q}$,  and all the orthogonal unit geodesics from the plaque minimize the distance from the plaque at least in the interval $[0,\epsilon]$.  When necessary, we will denote it as $\Om(P_q,\epsilon)$. Observe that if  $\Om(P_q,\epsilon)$ of $P_q$ is a tubular neighborhood then $\Om(P_q,\epsilon')$  is also a tubular neighborhood of $P_q$ for every $\epsilon'<\epsilon$. For every $x\in U$, we will denote by $P_x$ the plaque obtained as the connected component of $L_x\cap U$ which contains $x$. 
\end{definition}
\begin{remark}
Note that the plaque $P_q$ does not have to be a ``small'' submanifold, but it could be even a compact submanifold. We will need to reduce the plaque when we stress the differential structure, see for example part $(iv)$ of Proposition~\ref{proposition-submer}.
\end{remark}
\begin{remark}\label{reversetubes}
  For every point $x\in U$ in a tubular neighborhood $U$ (resp. a reverse tubular neighborhood $\tilde{U}$),  there exists a unique geodesic  $\gamma^x_+:[0,r_1]\rightarrow U$  from $P_q$ to $x$ (resp. $\gamma^x_-:[-r_2,0]\rightarrow \tilde U$ from $x$ to $P_q$)   which is orthogonal to $P_q$.
It is also possible to consider a reverse tubular neighborhood $\tilde U$ of $P_q$, namely, a tubular neighborhood for the reverse metric $\tilde F$, defined as $\tilde F(v)=F(-v)$ for $v\in TM$, with analogous properties.  When convenient, we will denote it by $\tilde \Om(P_x,\epsilon)$ in analogy with the straight case. 
\end{remark}

\begin{definition}
Let $\F=\{L\}$ be a singular  foliation on a  Finsler manifold $(M,F)$, $P_{q}$ be a plaque, and $U$ a tubular neighborhood of $P_q$ . 
Then the function $f_+:U\to [0,+\infty)$ given by $f_{+}(x):=d(P_{q},x)$ is continuous on its domain and smooth on $U\setminus P_q$. Analogously, if $\tilde U$ is a reverse  tubular  neighborhood of $P_q$, $f_-:\tilde U\to [0,+\infty)$, given by $f_{-}(x)=d(x,P_{q})$, is continuous on its domain and smooth on $\tilde U\setminus P_q$. 
Moreover, define the
\emph{future (resp. past)} cylinder  $\mathcal{C}^{+}_{r}(P_{q})$ (resp. $\mathcal{C}^{-}_{r}(P_{q})$) with axis $P_{q}$  as the level set  of $f_{+}$ (resp. $f_{-}$) for $r\in {\rm Im}(f_+)\setminus \{0\}$ (resp. $r\in {\rm Im}(f_-)\setminus\{0\}$). 
In other words $\mathcal{C}^{+}_{r}(P_{q})=f_{+}^{-1}(r)$ and $\mathcal{C}^{-}_{r}(P_{q})=f_{-}^{-1}(r)$. As $U$ and $\tilde U$ are tubular neighborhoods, $\mathcal{C}^{+}_{r}(P_{q})$ and $\mathcal{C}^{-}_{r}(P_{q})$ are smooth hypersurfaces.
\end{definition}
\begin{lemma}\label{remark-gamma-orthogonal-to-tubes}
With the above notation, given a plaque $P_q$ and a point $x$ in a future cylinder $\mathcal{C}^{+}_{r_1}(P_{q})$ (resp.  past cylinder $\mathcal{C}^{-}_{r_2}(P_{q})$), there exists a unique segment of geodesic $\gamma^x_+:[0,r_1]\to U$ (resp.  $\gamma^x_-:[-r_2,0]\to \tilde U$) from $P_q$ to $x$ (resp. from $x$ to $P_q$)  orthogonal to $P_q$ and $\mathcal{C}^{+}_{r_1}(P_{q})$ (resp. $\mathcal{C}^{-}_{r_2}(P_{q})$).
\end{lemma}
\begin{proof}
Let us prove the result for future cylinders as for past cylinders the proof is analagous using the reverse metric. If $x\in \mathcal{C}^{+}_{r_1}(P_{q})$, in particular $x$ belongs to the tubular neighborhood $U$, and then there exists a unique minimizing geodesic $\gamma^x_+$ from $P_q$ to $x$ of length equal to $r_1$. By the definition of  $\mathcal{C}^{+}_{r_1}(P_{q})$, $\gamma^x_+$ also minimizes the distance from $P_q$ to $\mathcal{C}^{+}_{r_1}(P_{q})$. Moreover, 
by \cite[Remark 2.1 (ii)]{CaJaMa}, it follows that $\gamma^x_+$ is orthogonal to $P_q$ and $\mathcal{C}^{+}_{r_1}(P_{q})$.
\end{proof}

\begin{definition}
We will say that a singular foliation is \emph{ locally forward (resp. backward) equidistant} if given a plaque $P_q$, a tubular neighborhood $U$ (resp. a reverse tubular neighborhood $\tilde U$) of $P_q$ and a point $x\in U$ (resp. $x\in \tilde U$) which belongs to the future cylinder  $\mathcal{C}^{+}_{r_{1}}(P_{q})$ (resp. the past cylinder $\mathcal{C}^{-}_{r_{2}}(P_{q})$),  then the  plaque
 $P_{x}\subset U$ (resp. $P_x\subset \tilde U$)  is contained in  $ \mathcal{C}^{+}_{r_{1}}(P_{q})$ (resp.  $\mathcal{C}^{-}_{r_{2}}(P_{q})$).
\end{definition}

%%%%%%%%%%%%
\begin{lemma}
\label{lemma-equidistant}
A  singular  foliation $\F$  is Finsler if and only if its leaves are locally forward and backward equidistant.
\end{lemma}
\begin{proof}

Lemma \ref{remark-gamma-orthogonal-to-tubes}
 allows us to check that if the leaves of $\F$ are locally equidistant,  then $\F$ is a singular Finsler foliation.

Now assume that $\F$ is a singular Finsler foliation, $P_q$ a plaque  and  $U$, a  tubular neighbohood of $P_q$. Given   $x_{0}\in U$, 
since $\F$ is a singular Finsler foliation, 
for each $x\in P_{x_{0}}$,  $\gamma^{x}_{+}$ is orthogonal to $P_{x_{0}}$ at $x$. On the other hand, by   Lemma \ref{remark-gamma-orthogonal-to-tubes}, $\gamma^{x}_{+}$
is also orthogonal to $\mathcal{C}^{+}_{r}(P_{q})$ at $x=\gamma^{x}_{+}(r)$. Therefore $T_{x}P_{x_{0}}\subset T_{x}\mathcal{C}^{+}_{r}(P_{q})$. Since this holds
for each $x\in P_{x_{0}}$, we conclude that $P_{x_{0}}$ must be contained in a future cylinder.  Therefore the leaves of $\F$ are locally forward equidistant.  Analogously, one can check that it is locally backward equidistant using a reverse tubular neighborhood.
\end{proof}

\begin{definition}[Homothetic transformations]
Consider a plaque $P_{q}$ of a singular Finsler foliation $\F$ and $\Om(P_q,\epsilon)$ a tubular neighborhood of $P_{q}$. Then for
each  $\lambda \in (0,1)$  (resp. $\lambda>1$)  we can define the \emph{future homothetic transformation} $h^{+}_{\lambda}:\Om(P_q,\epsilon)\to \Om(P_q,\lambda\epsilon)$  (resp. $h^{+}_{\lambda}:\Om(P_q,\frac{\epsilon}{\lambda})\to \Om(P_q,\epsilon)$)  as 
$h^{+}_{\lambda}(x)=\gamma^{x}_{+}(\lambda r)$ where $\gamma^{x}_{+}(r)=x$.
In particular $h_{\lambda}^{+}:\mathcal{C}^{+}_{r}(P_{q})\to \mathcal{C}^{+}_{\lambda r}(P_{q})$.
In a similar way, given a reverse tubular neighborhood $\tilde \Om(P_q,\epsilon)$ of $P_q$,  we define the \emph{past homothetic transformation } $h^{-}_{\lambda}:\tilde \Om(P_q,\epsilon)\to \tilde \Om(P_q,\lambda\epsilon)$ (resp. 
$h^{-}_{\lambda}:\tilde \Om(P_q,\frac{\epsilon}{\lambda})\to \tilde \Om(P_q,\epsilon)$)
as $h^{-}_{\lambda}(x)=\gamma^{x}_{-}(-\lambda r)$, where  $x=\gamma^{x}_{-}(-r)$. In particular $h_{\lambda}^{-}:\mathcal{C}^{-}_{r}(P_{q})\to \mathcal{C}^{-}_{\lambda r}(P_{q})$.
\end{definition}

\begin{lemma}[Homothetic transformation Lemma]
\label{lemma-homotheticlemma}
The future and past homothetic transformations $h_{\lambda}^{+}$ and $h_{\lambda}^{-}$  
send plaques to plaques of a singular Finsler foliation.
\end{lemma}
\begin{proof}

We will prove the lemma for the future  homothetic transformation $h_{\lambda}^{+}$. A similar proof is valid for 
the past homothetic transformation.

Let $x_{0}$ be a point in the future cylinder $\mathcal{C}_{r_{0}}^{+}(P_{q})$        with $r_0<\epsilon$.      For each $x\in P_{x_{0}}$ consider
$\gamma^{x}_{+}:[0,r_0]\to U$ the segment of  geodesic defined in Lemma \ref{remark-gamma-orthogonal-to-tubes} minimizing the distance from $P_q$ to $x$.

First we will consider  $\lambda<1$ close enough to 1. In particular, we will assume that there exists a reverse tubular neighborhood $\tilde V$ of $P_{x_0}$  (recall Remark \ref{reversetubes})  and the past cylinder $\mathcal{C}_{r_{2}}^{-}(P_{x_{0}})\subset \tilde V$ for $r_2=(1-\lambda)r_0$. Therefore,  $h_\lambda^+(x_0)\in\tilde V$ and  let ${\tilde P}_{h_{\lambda}^{+}(x_{0})}$ be the plaque of $h_\lambda^+(x_0)$ in $\tilde V$.    
Then,  by Lemma \ref{lemma-equidistant}, 
\begin{equation}
\label{eq2-lemma-homotheticlemma}
{\tilde P}_{h_{\lambda}^{+}(x_{0})}\subset\mathcal{C}_{r_{2}}^{-}(P_{x_{0}}),
\end{equation}
\begin{equation}
\label{eq3-lemma-homotheticlemma}
P_{h_{\lambda}^{+}(x_{0})}\subset\mathcal{C}_{r_{1}}^{+}(P_{q}),
\end{equation}
where $r_{0}=r_{1}+r_{2}$. In other words $r_{2}=(1-\lambda) r_{0}$ and $r_{1}=\lambda r_{0}$. For $x_\lambda\in \tilde P_{h^+_{\lambda}(x_{0})}\cap P_{h^+_{\lambda}(x_{0})}$, there exists  a segment of geodesic $\gamma_{2}$ with unit velocity so that
$\gamma_{2}(r_{1})=x_{\lambda}$, $\gamma_{2}(r_{0})\in P_{x_0}$ and $\dot{\gamma}_{2}(r_{0})$ is orthogonal to $P_{x_{0}}$; see
 (\ref{eq2-lemma-homotheticlemma}) and recall Lemma \ref{remark-gamma-orthogonal-to-tubes}.  Let $\gamma_{1}$ be the segment of geodesic with 
$\gamma_{1}(0)\in P_{q}$, $\gamma_{1}(r_{1})=x_{\lambda}$ and $\dot{\gamma}_{1}(0)$  orthogonal to $P_{q}$; see 
(\ref{eq3-lemma-homotheticlemma})  and  Lemma \ref{remark-gamma-orthogonal-to-tubes}. 
Let $\tilde{\gamma}$ be the (possible broken) geodesic constructed as the concatenation  of $\gamma_1$ with $\gamma_2$. Then we have that $\tilde{\gamma}$ 
 joins   $P_{q}$ to $\tilde\gamma(r_1+r_2)=\tilde\gamma(r_0)\in P_{x_0}\subset {\mathcal C}_{r_0}^+(P_q)$, it
is orthogonal to $P_q$ and has length $r_{0}=r_{1}+r_{2}$. Therefore $\tilde{\gamma}$ must be the geodesic $\gamma^{x}_{+}$, where $x=\tilde\gamma(r_1+r_2)$ and        $h_{\frac{1}{\lambda}}^+(x_\lambda)=x$.      This proves that   $h_{\frac{1}{\lambda}}^+(\tilde P_{h^+_{\lambda}(x_{0})}\cap P_{h^+_{\lambda}(x_{0})})\subset P_{x_0}$. In particular, 
\begin{equation}\label{ineq}
\dim P_{h^+_{\lambda}(x_{0})}\leq \dim P_{x_0}
\end{equation} 
(recall that $h^+_\lambda$ and $h^+_{\frac 1\lambda}$ are        inverse diffeomorphisms).      Proceeding analogously, one can prove that $h_{\frac{1}{\lambda}}^+(\hat P_{h^+_{\lambda}(x_{0})}\cap P_{h^+_{\lambda}(x_{0})})\subset P_{x_0}$ (where $\hat P_{h^+_{\lambda}(x_{0})}$ is the plaque of $h^+_\lambda(x_0)$ in a tubular neighborhood of $P_{x_0}$)  and  \eqref{ineq}  holds for $\lambda>1$ small enough, obtaining that $\dim P_{x_0}$ is a local maximum for the curve $\lambda\to\dim P_{h^+_{\lambda}(x_{0})}$. As $x_0$  is arbitrary, choosing $h_\lambda(x_0)$ as departing point, applying \eqref{ineq}, and taking into account that $h^+_\lambda\circ h^+_\rho=h^+_{\lambda\rho}$, it follows that  $\dim P_{h^+_{\lambda}(x_{0})}$ is a local maximum for every $\lambda$, and then it is constant with respect to $\lambda$.  This implies that  the following  property holds
\begin{itemize}
 \item[(P)] given $\lambda>0$ and $x_0\in U\setminus P_q$, there exists a neighborhood $\Omega\subset P_{x_0}$  of $x_0$ such that   $h^+_\lambda(\Omega)\subset P_{h_{\lambda}^{+}(x_{0})}$, and $h^+_\lambda|_\Omega$ is a diffeomorphism in the image,
\end{itemize} 
holds for every $x_0\in U\setminus P_q$ in $\lambda\in  [\mu,\nu]$ for a certain $\mu<1$  and $\nu\in (1,\epsilon/d(P_q,x_0))$ close enough to $1$  (which,        in principle,      depend on $x_0$).

 Now let us see that the property (P) for a fixed $x_0$ holds for every $\lambda\in (0,1]$. Let $\rho=\inf\{\mu\in (0,1]: \text{ (P) holds for all $\lambda\in [\mu,1]$}\}$. Then if $\rho>0$, there are two possibilities:  either $\rho$ satisfies (P) or not.  In the first case, as $h^+_{\eta \rho}=h^+_\eta\circ h^+_\rho$, and we know that $h^+_\rho(x_0)$ satisfies property (P) for all  $\eta<1$ close enough to $1$, it follows that all $\lambda$ in $[\eta\rho,1]$ also satisfies property (P), but $\eta\rho<\rho$, a contradiction.
Now if $\rho$ does not satisfy (P), consider $\eta>1$ satisfying (P) for $h^+_\rho(x_0)$. Then $h^+_{1/\eta}$ is well-defined in a neighborhood of $h^+_\rho(x_0)$ in $P_{h^+_\rho(x_0)}$         as $1/\eta<1$,       and the composition $h^+_{1/\eta}\circ h^+_{\eta\rho}=h^+_\rho$ shows that $\rho$ satisfies (P). This concludes that (P) holds for all $\lambda\in (0,1]$,  and for $\lambda\in (1,\epsilon/d(P_q,x_0))$, a similar reasoning works. 

Finally, observe that the subset $\mathcal A$ of points in the plaque $P_{x_0}$ which are sent by $h^+_\lambda$ to the plaque $P_{h^+_\lambda(x_0)}$ is open by the property (P). Moreover, given a point in the closure of this subset, property (P) shows that it is in the interior and then $\mathcal A$ is closed. Being the plaque connected, it follows that $\mathcal A$  is the whole plaque $P_{x_0}$ as required.  As this also holds for the inverse $h^+_{1/\lambda}$, we conclude that $h^+_\lambda: P_{x_0}\to P_{h^+_\lambda(x_0)}$ is a diffeomorphism. 
\end{proof}
%%%%%%
Let us now review  a useful and standard  lemma. 

\begin{lemma}
\label{lemma-subfoliation}
Let $\F$ be a singular foliation on $M$. Then for $q\in M$ there exists a neighborhood $U$ of $q$, a regular 
foliation $\F^{2}$  on $U$ (i.e., all the leaves of $\F^{2}$  have the same dimension) 
and an embedded submanifold $S_{q}$ (a slice) so that:
\begin{enumerate}
\item[(a)] $\F^{2}$ is a subfoliation of $\F$, i.e.,  for each $x\in U$, the leaf $L^{2}_{x}$ of  $\F^{2}$ through $x$ 
is contained in the leaf $L_{x}$ of $\F$;
\item[(b)] $L^{2}_{q}$ is a relatively compact  open subset of $L_{q}$ (and in particular it has the same dimension  as  $L_{q}$)
\item[c)] $S_q$ is transverse to $\F^{2}$, i.e., $T_{x}M= T_{x}S_{q}\oplus T_{x}L_{x}^{2}$ for each $x\in S_{q}$,  $q\in S_q$ and $S_q$ meets all the plaques of $\F^2$ in $U$. 
\end{enumerate}
 \end{lemma}
 \begin{definition}  Given a tubular neighborhood  $\Om(P_q,\epsilon)$ of a plaque $P_q$,
  in a singular Finsler foliation $(M,\F,F)$, we can define the {\em projection map} $\rho:\Om(P_q,\epsilon)\rightarrow P_q$ as $\rho:=h_0^+$, namely, the homothetic transformation $h^+_\lambda$  in the limit case of $\lambda=0$,        or, more precisely, $\rho^+(x)=\gamma^x_+(0)$ for $x\in \Om(P_q,\epsilon)\setminus P_q$ and $h^+(x)=x$ if $x\in P_q$.      Moreover, given $p\in P_q$, we define the {\it Finslerian slice} $\Lambda_p$ as the image by the exponential map of $\nu(P_q)\cap T_{p}M\cap F^{-1}(0,\epsilon)$. 
\end{definition} 
\begin{proposition}\label{proposition-submer}
     Given a point $q\in M$, a plaque $P_q$ of $q$ and a tubular neighborhood $\Om(P_q,\epsilon)$, the following properties hold:
\begin{enumerate}[(i)]
\item for every $p\in P_q$, the Finslerian slice $\Lambda_p$ is transversal to all the plaques that meets, 
\item  the projection map $\rho:\Om(P_q,\epsilon)\rightarrow P_q$ is a  surjective submersion,
\item given any plaque $P_x$ in $\Om(P_q,\epsilon)$, the restriction  $\rho|_{P_x}:P_x\rightarrow P_q$ is a  submersion and, when the leaves of $\F$ are closed, $\rho|_{P_x}$ is surjective, and
\item by reducing $P_q$ and $\epsilon$ if necessary, we can assume that $\rho|_{P_x}$ is surjective.
\end{enumerate}  
\end{proposition}
\begin{proof}
     For part $(i)$, given $y\in \Lambda_p$, consider the unique $v\in \nu_{\rho(y)}(P)$ with $F(v)<\epsilon$ such that $\exp(v)=y$. Consider a basis $\{e_1,\ldots, e_r\}$ of $T_{\rho(y)}P_q$ and a regular subfoliation $\F^2$ of the foliation $\F$ in a neighborhood $U$ of $\rho(y)$ as in Lemma \ref{lemma-subfoliation}. Then  by reducing $U$ if necessary, we can extend $\{e_1,\ldots, e_r\}$ to a frame $\{X_1,\ldots,X_r\}$ of the leaves of $\F^2$. Consider now a basis $\{Y_1,\ldots,Y_{n-r}\}$ of the $g_v$-orthogonal space to $T_{\rho(y)}P_q$. Observe that defining $Z_i(t)=d\exp_{tv}[Y_i]$, for $t\in(0,1]$ and $i=1,\ldots,n-r$ and using Proposition \ref{proposition-tnu}        and that $g_v=g_{tv}$ for $t>0$,      it follows that $\{Z_1(t),\ldots,Z_{n-r}(t)\} $ is a basis of $
T_{\exp(tv)}\Lambda_p$. Consider the system 
\[\{X_1(\exp(tv)),\ldots,X_r(\exp(tv)),Z_1(t),\ldots,Z_{n-1}(t)\}\]
 for $t\in[0,1]$. As in $t=0$ this system is linearly independent, by continuity, it will be linearly independent for small $t$, which implies that $\Lambda_p$ is transversal to the plaques of $\F$ in $\exp(tv)$        for such a small $t$.      As the slice $\Lambda_p$ is invariant by the homothetic transformation, applying the homothetic transformation lemma, it follows that $\Lambda_p$ is also transversal in $y=\exp(v)$, proving part $(i)$. Part $(ii)$ follows from the observation that $\rho$ is the composition of a diffeomorphism with $\tau:\nu(P_q)\cap F^{-1}(\epsilon)\rightarrow P_q$, which is a submersion by Proposition \ref{proposition-nuP}. The surjectivity of $\rho$ follows from definition. For part $(iii)$, observe that  the kernel of $\rho$  is the tangent space to the slices $\Lambda_p$. As $\Lambda_p$ is transversal to the         plaques,      $d(\rho|_{P_x})$ has to be surjective by a counting of dimensions.

     In order to see that $\rho|_{P_x}$ is surjective, observe that its image $\rho(P_x)$ is open and closed in $P_q$. Open because        $\rho|_{P_x}$      is a submersion. Let us see that it is closed. If $\{u_n\}_{n\in\N}\in \rho(P_x)$ is a sequence that converges to $u\in P_q$,  let $\gamma^{u_n}:[0,b]\rightarrow M$ be a  unit minimizing geodesic from $u_n$ to $P_x$ ($b\in\RR$ does not depend on $n$). Then $\{\dot \gamma^{u_n}(0)\}_{n\in\N}$ admits a subsequence in $\nu^1(P_ q)$ converging to $v\in \nu^1_u(P_q)$. Let $\gamma_v:[0,b]\rightarrow M$ be the geodesic such that $\dot\gamma_v(0)=v$. It turns out that $\gamma_v([0,b])\subset \Om(P_q,\epsilon)$, $\gamma_v(b)\in P_x$, because $P_x$ is closed,  and  $u=\rho(\gamma_v(b))$, as required. 
%For part $(ii)$, observe that 
%\begin{align*}
%\dim (T_y\Lambda_{\pi(y)}+T_yP_x)&=\dim \Lambda_{\pi(y)}+\dim P_x-\dim(T_y\Lambda_{\pi(y)}\cap T_yP_x)\\
%&=\dim M-\dim P_q+\dim P_x-\dim(T_y\Lambda_{\pi(y)}\cap T_yP_x)
%\end{align*}
%and $\dim(T_y\Lambda_{\pi(y)}\cap T_yP_x)=\dim Y\geq -\dim P_q+\dim P_x$ as seen above. This implies that $\dim (T_y\Lambda_{\pi(y)}+T_yP_x)\geq \dim M$ and therefore $\Lambda_{\pi(y)}$ is transversal to $P_x$. 
Part $(iv)$ follows from part $(iii)$ observing that it is possible to reduce $P_q$ and $\epsilon$ in such a way that there exists a regular subfoliation $\F^2$ with closed leaves, and then, analogously to part         $(iii)$,      the restriction of $\rho$ to the plaques of $\F^2$ is also a surjective submersion.
   
%consider  the submanifold $\Lambda$ obtained as the image by the exponential map of the subset $\Omega=\{v\in \nu P_q\cap T_{\pi(y)}M: F(y)<\epsilon\}$. As the exponential is a diffeomorphism in this subset (recall the definition of $O(P_q,\epsilon)$). Then to prove that $\pi$ is a submersion is equivalent to prove that $\Lambda$ is transversal to $P_x$.   
\end{proof}
The above      results    can be used to prove the next proposition.

\begin{corollary}\label{equivalencedef}
Let $(M,F,\mathcal F)$ be a (regular) Finsler foliation with connected fibers given by the leaves of a submersion  $\pi:M\to N$. Then there exists a unique Finsler metric $\hat F$ on $N$ such that $\pi:(M,F)\to (N,\hat F)$ is a Finsler submersion. As a consequence, every        regular      Finsler foliation  is described locally by a Finsler submersion.
\end{corollary}
\begin{proof}
Given $p\in M$, consider the map $\hat F_{\pi(p)}:T_{\pi(p)}N\to \mathbb{R}$ given by
\begin{equation}
\hat F_{\pi(p)}:=F_p\circ \varphi_p^{-1},
\end{equation}
where $\varphi^{-1}_p$ is the horizontal lift of $d \pi$ (see part (a) of Lemma~\ref{lem:exMin}). Observe that, for every $p\in M$, $\hat F_p=F_p\circ \varphi_p^{-1}$ is a      Minkowski norm    such that $d\pi_p:(T_pM,F_p)\to (T_{\pi(p)}N,\hat F_{\pi(p)})$ is a      linear    Finsler submersion (see part (b) of Lemma~\ref{lem:exMin}). Let us see that this map is well-defined, namely,  $F\circ \varphi^{-1}_{p_1}=F\circ \varphi^{-1}_{p_2}$, for every $p_1,p_2\in \pi^{-1}(q)$. 
     
%Observe that for every $p\in M$, $\hat F_p=F_p\circ \varphi_p^{-1}$ is a Finsler metric such that $d\pi_p:(T_pM,F_p)\to (T_{\pi(p)}N,\hat F_{\pi(p)})$ is a Finsler submersion (see part (b) of Lemma~\ref{lem:exMin}). Let us see that this map is well-defined, namely,  $F\circ \varphi^{-1}_{p_1}=F\circ \varphi^{-1}_{p_2}$, for every $p_1,p_2\in \pi^{-1}(q)$. 
As the fibers are connected, it is enough to prove that this property is open.              Therefore we can assume that $p_1$ and $p_2$ belong to  the same plaque,  say $P_1:=P_{p_1}$. Given $v\in T_qN$, let $v_i=\varphi_{p_i}^{-1}(v) \in T_{p}M$ with $i=1,2$. We will show that $F(v_1)=F(v_2)$. Consider        $\Om(P_1,\epsilon)$      a tubular neighborhood of $P_1$ and fix $P$ a plaque in  $\Om(P_1,\epsilon)$ which contains $\gamma_{v_1}(s)$ for some $s$.  Let $w \in \nu_{p_2}(P_1)$ be such that $\gamma_{w}|_{[0,s]}$  attains the distance $d_F(P_1,P)$        (recall part (iii) of Proposition \ref{proposition-submer}).      In particular, $F(v_1)=F(w)$. By the homothetic lemma \ref{lemma-homotheticlemma}, we get that $\pi(\gamma_{v_1}(t))=\pi(\gamma_{w}(t))$ for $t$ in a neighborhood of $0$ and then $d\pi(w)=d\pi(v_1)=d\pi(v_2)$. Therefore $w=v_2$ by the injectivity of  $\varphi_{p_2}$, and consequently $F(v_1)=F(v_2).$  
The  smoothness and unicity  of $\hat F$ follow from Lemmas~\ref{lem:exMin} and \ref{lem:exMin2}.

\end{proof}

%%%%%%%%%%%%%%%%%%%%%%%%%%%%%%%%%%%%%%%%%%%%%%%
\begin{remark}\label{rem:transport}
     Recall that given a Finsler submersion  $\pi:(M,F)\to (N,\hat F)$, each segment of geodesic in $(N,\hat F)$ can be lifted to a
horizontal segment of geodesic in $(M,F)$, recall \cite[Theorem 3.1]{Duran}. This fact together with the above corollary allow us 
to \emph{transport} horizontal segments of geodesics along regular leaves. More precisely, let  
$\F$ be a singular Finsler foliation on a complete Finsler manifold $(M,F)$,  $\gamma:[0,1]\to M$
be a horizontal segment of geodesic such that $\gamma|_{(0,1]}$ has only regular points and $\beta:[0,1]\to L_{\gamma(1)}$ a curve with
$\beta(0)=\gamma(1)$. 
Then there exists a variation  of horizontal segments of  geodesics $s\to\gamma^{s}$ 
such that $\gamma^{0}=\gamma$, $\gamma^{s}(1)=\beta(s)\in L_{\gamma(1)}$, $\gamma^{s}|_{(0,1]}$ has only regular points and $\pi(\gamma^{s})=\pi(\gamma)$, where $\pi:M\to M/\F$ is the canonical projection. Due to the homothetic lemma
we can also infer that $\gamma^{s}(0)\in L_{\gamma(0)}$ even if  $L_{\gamma(0)}$ is a singular leaf.         Analogously, the transport can be done from the initial leaf $L_{\gamma(0)}$ to $L_{\gamma(1)}$ assuming that $\gamma_{[0,1)}$ has only regular points.       

\end{remark}

%%%%%%%%%%%%%%%%%%%%%%%%%%%%%%%%%%%%%%%%%%%%%%%%%%% NEWCOMANND NO MEIO DO TEXTO???
%\newcommand{\sq}{{q}}
     Recall that a {\it transnormal function} $f:M\rightarrow \RR$ of a Finsler manifold $(M,F)$ is a function such that there exists another  real function $b:f(M)\rightarrow \RR$ such that  $F(\nabla f)^2=b(f)$, where $\nabla f$ is the gradient of $f$, namely, $\nabla f=\mathcal L^{-1}(df)$.  
%\footnote{       tal vez podemos colocar $I$ no dominio de b e depois dizer que $f(M)\subset I$     }
   
\begin{proposition}
\label{proposition-metricFS}
Let $\F=\{L\}$ be a singular Finsler foliation on a  Finsler manifold $(M,F)$. 
Given a point $q\in M$,  there exists a slice $S_q$ transversal to $L_q$  and a  Finsler metric $\hat F$ on 
$S_q$  with the following properties:
\begin{enumerate}
\item[(a)]
 the leaves of 
the slice foliation $\F_{q}=\F\cap S_{q}$  endow $(S_q,\hat F)$ of a structure of 
singular Finsler foliation,   
\item[(b)]   the distance between the leaves of $\F_{q}$ 
(with respect to $\hat F$)  
 and the distance between the leaves of $\F$ (with respect to $F$) coincide locally. 
 
\end{enumerate}
\end{proposition}
\begin{proof}
%%% 
The proof is       adapted from the one of    the Riemannian version  \cite[Proposition 2.2]{AlexToeben2}. 
Let us start by constructing  the desired metric $\hat F$.    

Let $\F^{2}$ be the subfoliation of $\F$ with plaque $P_{q}=L^{2}_{q}$ 
in a neighborhood $U$  and $S_q$ be the slice  presented  in Lemma  \ref{lemma-subfoliation}.

%Given $x\in S_q$, observe that $\F^2$ meets the slice transversely in only one point at each leaf and as a consequence $T_xM=T_xL^2_x\oplus %T_xS_q$ and 

Using  Lemma \ref{lemma-subfoliation}      and reducing $U$ if necessary,    we can define  a submersion $\pi:U\rightarrow S_q$ whose fibers are the leaves of $\F^2$. 
 By part (b) of Lemma \ref{lem:exMin2}, there exists a Finsler metric $\hat F$      on $S_q$    such that 
\begin{equation}\label{lin-sub}
d \pi_{x}: (T_xM,F_x)\rightarrow (T_xS_q,\hat F_x)
\end{equation}
is a linear  Finsler submersion      for every $x\in S_q$.     Fix $p\in S_q$ and  let $P_p$ be the plaque at $p$ in $U$. Consider now a tubular neighborhood $\tilde U$ of the plaque $P_p$ and define $f_+:\tilde U\setminus P_p\rightarrow \RR$ as $f_+(x)=d^+(P_p,x)$, which is a transnormal function, and let $X$ be the $F$-gradient of $f_+$. Consider also the restriction $\hat f_+=f_+|_{\tilde U\cap S_q\setminus (P_p\cap S_q)}$ and the vector field along $S_q$, $\tilde X=d \pi(X)$. Observe that the fibers of the linear submersions $d\pi_{x}$ in \eqref{lin-sub} are contained in the tangent spaces of the cylinders with axis   the plaque $P_p$ because these cylinders contain the plaques of the foliation $\F$ and then also those of $\F^2$      (see Lemma \ref{lemma-equidistant}).   
 This implies that $X$ is $F$-orthogonal to the fibers of $\pi$. As a consequence, $F(X)=\hat F(\tilde X)$ and $ d\pi_{x}: (T_xM,g_X)\rightarrow (T_xS_q,\hat g_{\tilde X})$ is a linear Riemannian submersion by \cite[     Prop. 2.2  ]{Duran}, where $g$ is the fundamental tensor of $F$ and $\hat g$ the fundamental tensor of $\hat F$. In particular,  $g_X(X,v)=\hat g_{\tilde X}(\tilde X, d\pi_{x}(v))$ for all $v\in T_xM$. Given $v\in T_xS_q$, we have $d\hat f_+(v)=df_+(v)=g_X(X,v)=\hat g_{\tilde X}(\tilde X,v)$ and consequently $\tilde X$ is the $\hat F$-gradient of $\hat f_+$;
recall Remark \ref{remark-legrendre-transformation}. Moreover, $\hat F(\tilde X)=F(X)=1$, which implies that $\hat f_+$ is a transnormal function of $(\tilde U\cap S_q,\hat F)$, which        extends to zero       on $P_p\cap S_q$.        Observe that part $(2)$ of 
 \cite[Proposition 4.1]{He-Yin-Shen} does not require any kind of completeness. Then applying this result to $\hat f_+$,  it follows that      the         integral lines      of  the $\hat F$-gradient of $\hat f_+$ are segments of geodesics, and it is not difficult to check that these geodesics meet  $P_p\cap S_q$ orthogonally.          Considering a tubular neighborhood $\tilde V$ of $P_p\cap S_q$ for $(\tilde U\cap S_q,\hat F)$,     
it follows that the level sets of $\hat f_+|_{\tilde V}$ coincide with the future cylinders of $P_p\cap S_q$ with respect to $\hat F$ and then the plaques of $\F\cap S_q$ in         
$\tilde V$      are contained in these cylinders as the plaques in $\tilde U$ of the foliation $\F$ are contained in the level sets of $f_+$.  Proceeding analogously with the reverse metrics of $F$ and $\hat F$, we deduce that the leaves of $\F\cap S_q$ are contained locally in the past cylinders of  $\hat F$, concluding that $(\F\cap S_q,\hat F)$ is a singular Finsler foliation.
The fact that $\hat F(\tilde X)=F(X)=1$ also allows us to conclude that  locally the distances between the leaves coincide with those of $(M,       F,     \F)$.

\end{proof}
\begin{remark}
     Observe that the meaning of locally in part (b) of Proposition \ref{proposition-metricFS} as it can be checked in the last proof is the following:  given a plaque $P_p$ of $U$ with $p\in S_q$ as in the last proof, there exists a tubular
neighborhood $\tilde U$        of $P_p$ for $(M,F)$ and another tubular neighborhood $\tilde V$ of $P_p\cap S_q$ for $(\tilde U\cap S_q,\hat F)$ such that given $x\in \tilde V$, if $\tilde P_x$ is the plaque of $x$ in $\tilde V$,  the $\hat F$-distance from $ P_p\cap S_q$ to $\tilde P_x$ coincides with the $F$-distance from $P_p$ to  $P_x$.     
   
\end{remark}
\begin{definition}\label{locclosed}
We will say that a singular Finsler foliation $(M,\F,F)$ has {\it locally closed leaves} if, for every point $q\in M$, there exists a slice $S_q$ endowed with a structure of singular Finsler foliation as in Proposition \ref{proposition-metricFS} with closed (and then compact) leaves.
\end{definition}
\begin{proposition}[Slice reduction]
\label{proposition-sliceflat-metric}
Let $\F=\{L\}$ be  a singular Finsler foliation        with locally closed leaves      on a  Finsler manifold $(M,F)$. Then for every point $q\in M$, there exists a slice $S_q$ and a Finsler metric $\hat F$ in $S_q$ as in Proposition \ref{proposition-metricFS}, such that  $\F_{q}=S_{q}\cap \F$ with the metric $\hat F$  is foliated diffeomorphic
to        an open subset of  a Minkowski space  endowed with a structure of        singular Finsler foliation.
\end{proposition}
\begin{proof}

Let $\hat F$ be the Finsler metric on $S_{q}$ defined in Proposition \ref{proposition-metricFS}. Then 
the proof   turns out to be   similar to the proof of the analogous result for singular Riemannian foliations; see \cite[Propositon 2.3]{AlexToeben2}. In fact,  on the Finsler space $(S_{q}, \hat F)$  and reducing $S_q$ if necessary,  consider the homothetic transformation $h_{\lambda}: S_{q}\to S_{q}$ defined as $h^+_{\lambda}(x)=\gamma(\lambda r)$, where $\gamma$ is the radial (Finsler) geodesic on  $(S_{q},\hat F)$ starting at $q=\gamma(0)$ and $x=\gamma(r)$.   
Set ${\hat F}^{\lambda}:=\frac{1}{\lambda} h_{\lambda}^{*} \hat F$. 
Note that ${\hat F}^\lambda$ converges (uniformly) to $ \hat F^{0}$ when $\lambda\to 0$, 
where $\hat F^{0}$ is the Finsler metric on $S_{q}$ so that
$\hat F_{q}=\exp_{q}^{*}\hat F^{0}$ for the Minkowski norm ${\hat F}_{q}$ on $T_{q}S_{q}$.   

     Since $\F_{q}$ is a singular Finsler foliation on $(S_{q}, \hat F)$, it follows from Lemma \ref{lemma-homotheticlemma}  that
$\F_{q}$ is a singular Finsler foliation on $(S_{q}, {\hat F}^{\lambda})$.   

     We claim that \emph{$\F_{q}$ is a singular Finsler foliation on 
$(S_{q}, \hat F^{0})$}, reducing the slice $S_q$ if necessary. This will imply that  \emph{$\F_{q}$ is diffeomorphic to        an open subset of      a singular Finsler foliation on the Minkowski space $(T_{q}S_{q}, {\hat F}_{q})$} concluding the proof.            Observe that even if, in principle, the singular Finsler foliation is defined using the exponential map in an open subset $U$ of $T_qS_q$ of the zero vector, we can extend this foliation to the whole $T_qS_q$ by using the homothetic lemma, namely, the leaves of the extended foliation are determined by the images by the homothetic  transformations of the leaves in $U$, and this singular foliation is Finsler because the orthogonality is invariant by homotheties and geodesics are always straight lines.     

     In order to prove the claim we only need to check that, given  leaves $L,\tilde L\subset S_q$, and $x,y\in \tilde L$, there exists $\lambda_{0}$ so that 
\begin{equation}
\label{eq-slice-distance-lambda}
d_{\lambda}(L,x)=d_{\lambda}(L,y)
\end{equation}
for $\lambda< \lambda_{0}$, where $d_\lambda$ is the distance in $S_q$ computed with the metric $\hat F^\lambda$.

     Once we have proved \eqref{eq-slice-distance-lambda}, we can take the limit and conclude that 
\begin{equation}
\label{eq-slice-distance-lambda-2}
d_{0}(L,x)=d_{0}(L,y).
\end{equation}
  The claim will follow from  
\eqref{eq-slice-distance-lambda-2}, its analogous for the metric $v\to \widehat{F}^{0}(-v)$ 
and from Lemma \ref{lemma-equidistant}.   

     Since $\F_{q}$ is a singular Finsler foliation with respect to $\widehat{F}^{\lambda}$,  
\eqref{eq-slice-distance-lambda} is clearly  true if $x$ and $y$ are in a tubular neighborhood $U_\lambda$ of $L$. 
The problem is that the neighborhood $U_\lambda$ could be arbitrarily small when $\lambda$ goes to zero. The underlying reason for the apparent lack of control on $U_\lambda$ in the Finslerian setting is that  
it is not clear if $\{ \widehat{F}^\lambda  \}$  converges smoothly to $\widehat F^0$ (recall, unlike what happens in the Riemannian geometry setting,  $\exp$ is just $C^{1}$ at zero in Finsler geometry). We aim to prove \eqref{eq-slice-distance-lambda} for all $x$ and $y$ in $S_q$ (not only for those  contained  in a tubular neighborhood of  $L$). Here we will use compactness of the leaves of $\F_{q}$. In fact this is the only
place in this paper where we have used this assumption.    

     First note that there exists $R_2$, $0<R_{1}<R_{2}$ and $\lambda_{0}>0$ with the following properties:        the ball $B_{R_2}(q)\subset S_q$ of center $q$ and radius $R_2$ with respect to $\widehat F$ is precompact and,      if 
\emph{$\tilde{x}$ and $\tilde{y}$
are in $B_{R_{1}}(q)$         and $\lambda< \lambda_{0}$, then there exists a segment of 
geodesic $\gamma^\lambda$ (with respect to $\widehat{F}^{\lambda}$)  joining $\tilde{x}$ and $\tilde{y}$ such that 
$\ell_{\lambda}(\gamma^\lambda)=d_{\lambda}(\tilde{x},\tilde{y})$}, where $\ell_{\lambda}(\gamma^\lambda)$ is the length of $\gamma^\lambda$ with respect to $\widehat{F}^\lambda$. 
One can check this observation using the fact that $\widehat{F}^{\lambda}$ converges to $\widehat{F}^{0}$ uniformly.

       From now on, we will assume that $L,\tilde L\subset B_{R_1}(q)$. 
Thus proving \eqref{eq-slice-distance-lambda} for this data, we will conclude by reducing the slice $S_q$ to $B_{R_1}(q)$, since \eqref{eq-slice-distance-lambda-2} implies that all the leaves which are close enough to $L$ remain in an $\widetilde F^0$-tubular neighborhood of $L$ contained in $B_{R_1}(q)$.     
      First we consider the case in that $\tilde L$ is regular. Let $x_{0}\in \tilde L$ so that 
$d_{\lambda}(L,\tilde L)=d_{\lambda}(L,x_{0})$ and let $\gamma^{x_0}$ be a segment of geodesic so that  $\gamma^{x_{0}}(0)\in L$, $\gamma^{x_{0}}(1)=x_{0}$ and $\ell_{\lambda}(\gamma^{x_{0}})=d_{\lambda}(L,x_{0})$. 
\emph{Note that $\gamma^{x_{0}}|_{(0,1)}$ has only regular points.} 
In fact assume by contradiction that there exists $0<t_{1}<1$ so that $\gamma^{x_{0}}(t_{1})$ is singular and $\gamma^{x_{0}}(t)$ is
regular for $t>t_{1}.$ 
Note that, by the homothetic lemma and  Proposition  \ref{proposition-submer}, there exists a segment of geodesic $\beta|_{[t_1,t_1+\epsilon]}$ such that $\beta(t_1)=\gamma^{x_0}(t_1)$, $\beta(t_1+\epsilon)\in P_{\gamma^{x_0}(t_1+\epsilon)}$, for a small $\epsilon>0$, which has the same projection on the regular stratum as $\gamma^{x_0}|_{[t_1,t_1+\epsilon]}$. Since they have the same projection and $\gamma^{x_0}|_{(t_1,1]}$ lies in the regular stratum, then $\beta(1)\in P_{\gamma(1)}\subset \tilde L$ and 
$\ell_{\lambda}(\beta)=\ell_{\lambda}\big(\gamma^{x_{0}}|_{[t_{1},1]}\big)$.        Indeed, as the points of $\gamma^{x_0}_{[t_1+\epsilon,1]}$ are regular, $\beta|_{[t_1+\epsilon,1]}$ will be the transport the segment of geodesic $\gamma|_{[t_1+\epsilon,1]}$ (recall Remark \ref{rem:transport}).     
Therefore $\tilde{\gamma}=\beta * \gamma^{x_{0}}|_{[0,t_{1}]}$  is a minimal broken geodesic joining $L$ to $\tilde L$ so that
$\ell_{\lambda}(\tilde{\gamma})=\ell_{\lambda}(\gamma^{x_{0}})$, which is  absurd; compare with 
\emph{Kleiner's lemma} (see \cite[Lemma 3.70]{AlexBettiol}).   

     Once $\gamma^{x_{0}}|_{(0,1)}$ has only regular points, one can ``transport" 
this segment of geodesic, producing segments  $\gamma^{x}$
and $\gamma^{y}$ joining $L$ to $x$ and $y$ so that 
$\ell_{\lambda}(\gamma^{x})=\ell_{\lambda}(\gamma^{y})=\ell_{\lambda}(\gamma^{x_{0}})=d_{\lambda}(L,\tilde L)$. This last 
equation implies \eqref{eq-slice-distance-lambda} when $\tilde L$ is regular.
The case in that $\tilde L$ is singular  follows from the regular case and from fact that the set of regular points is dense on $S_q$
(which can be proved again using the Kleiner's Lemma type argument discussed above)        and using \eqref{eq-slice-distance-lambda-2} for the sequence of regular leaves approximating $\tilde L$.      
  
\end{proof}
%%%%%

For the next proposition we will need the following useful lemma. 
%%%%%%%%
\begin{lemma}
\label{lemma4-RandersFoliation-minkowskiSpace}
Let $\F$  be a singular Finsler foliation on a Minkowski space $V$. Assume
that $\{0\}=L_{0}$, i.e., zero is a leaf. Then 
the minimal stratum $\Sigma$ (i.e., the union of leaves with dimension zero) is  a subspace of %$\mathbb{R}^{n}$.
$V$.
\end{lemma}
\begin{proof}
 Observe that given a point in the minimal strata, it is possible to choose the whole %$\RR^n$ 
$V$ as a simple neighborhood. Then the result follows easily by  repeatedly applying Lemma \ref{lemma-homotheticlemma}. 
\end{proof}

\begin{proposition}[Stratification]
\label{proposition-stratification}
Let $\F=\{L\}$ be a singular Finsler foliation with locally closed leaves on a complete Finsler manifold $(M,F)$. Then the union of the leaves with the same dimension is a disjoint union of embedded submanifolds (called stratum).
The collection of all the strata is a stratification in the usual sense.  
In addition 
the induced  foliation $\F|_{\Sigma}$ restricted to a stratum $\Sigma$  is a (regular) Finsler foliaton.
 \end{proposition}
\begin{proof}
The fact that each connected component of a stratum is an embedded submanifold 
follows from Proposition \ref{proposition-sliceflat-metric} and Lemma \ref{lemma4-RandersFoliation-minkowskiSpace}.
The same results imply that the collection of all strata is a stratification; recall \cite[Definition 3.100]{AlexBettiol}.
Alternatively this can be proved in an analogous way as it was proved
in \cite[Proposition 6.3]{Molino} for singular Riemannian foliations.

Let us  see why $\F|_{\Sigma}$ is a regular Finsler foliation. 
Let $y\in \Sigma$ and consider a plaque $P_y$ which contains $y$ and a tubular neighborhood  $U$ of $P_{y}$.
Then,      for every $x\in U\cap\Sigma$,    there exists a minimal segment of horizontal geodesic $\gamma_{+}^{x}$ (contained  in $U$) joining the plaque $P_{y}$ to  $x$.  
By Lemma \ref{lemma-homotheticlemma},  $\gamma_{+}^{x}$ is contained in $\Sigma$. This implies that the transverse geometry of $\Sigma$ coincides with the transverse geometry of $M$. 
In particular the  future and past cylinders in $\Sigma$ coincide with the intersection of $\Sigma$ with the future and past cylinders with axis      contained    in $\Sigma$. The result now follows from Lemma \ref{lemma-equidistant}.  

\end{proof}

The next proposition can be adapted from  \cite[Proposition 6.4]{Molino}; see also \cite[Proposition 2.14]{Alex6}.

\begin{proposition}
\label{proposition-foliation-submersion}
Let $\F=\{L\}$ be a singular Finsler foliation with locally closed leaves on a complete Finsler manifold $(M,F_{1})$. Assume that there exists a complete Finsler metric $F_{2}$ such
that the foliation restricted to each stratum is a (regular) Finsler foliation with respect to $F_{2}$. Then $\F$ is a singular Finsler foliation
on $(M,F_{2})$.
\end{proposition}

We stress that  we don't need the above proposition in the proof of the main results, just the Riemannian case proved in \cite[Proposition 2.14]{Alex6}.

%%%%%%%%%%%%%%%%%%%%%%%%%%%%%%%%%%%%%%%%%%%%%%%%%%%%%%%%%%%%%%%%%%%%%%%
%%%%%%%%%%%%%%%%%%%%%%%%%%%%%%%%%%%%%%%%%%%%%%%%%%%%%%%%%%%%%%%%%%%%%%
\section{Singular Finsler foliations on Randers spaces}
\label{section-singular-Finsler-Foliation-Randers}
In this section we prove 
Theorem \ref{theorem-Randersfoliation}. For this purpose,  we apply the
slice reduction presented in Proposition \ref{proposition-sliceflat-metric} to
relate local singular Finsler foliations on Randers spaces  with  singular Finsler foliations on Randers-Minkowski spaces.

\subsection{Singular Finsler foliations on Randers-Minkowski spaces}
\label{sec-singularFinslerfoliation-RandesMinkowski}

In this section we present some facts about singular Finsler foliations $\F=\{L\}$ on a Randers Minkowski space $(\mathbb{R}^{n}, Z)$, where
$Z(v)=\sqrt{\langle v,v \rangle}+\beta(v)$ with  $\langle \cdot, \cdot \rangle$ denoting the Euclidean product and $\beta(v)=\langle v, \dualbeta \rangle$
for a  constant vector field $\dualbeta$ with length smaller than 1.
Recall that $\dualbeta$ is multiple of the associated wind; see Lemma \ref{lemma-Randes-equivalence}.
 We will assume that the minimal leaves (i.e, the leaves with the lowest dimension)
have dimension zero and that $\{0\}=L_{0}$, i.e., the zero is one of these minimal leaves.

The main goal of this section is to prove Lemma \ref{lemma6-RandersFoliation-minkowskiSpace}, which shows that the constant vector $\dualbeta$ (and hence the wind) is tangent to the minimal stratum $\Sigma$ (i.e., the union of minimal
leaves).

\begin{lemma}
\label{lemma1-RandersFoliation-minkowskiSpace}
For each leaf $L$ of $\F$ we have that the tangent space $TL$ is perpendicular to  the constant vector field $\dualbeta$  with respect to the Euclidean metric. In particular
we infer that $L\subset V$, where $V=\{y \in \mathbb{R}^{n}, \langle y, \dualbeta\rangle=c\}$ for some $c\in\RR$.
\end{lemma}
\begin{proof}

 From Lemma \ref{lemma-product-gvvu} we have:
\begin{equation}
\label{eq1-lemma1-RandersFoliation-minkowskiSpace}
\metric_{v}(v,w)=Z(v) \Big ( \frac{\langle v,w\rangle}{\|v\|}+ \langle w,\dualbeta \rangle \Big ).
\end{equation}
 Since $\F$ is a      singular    Finsler foliation then the straight lines $\RR\ni t\rightarrow t v\in\RR^n$ and $\RR\ni t\rightarrow (1-t) v\in\RR^n$  are (Finsler) orthogonal to the leaf $L_{v}$. Hence  
$\metric_{v}(v,w)=0$ and $\metric_{-v}(-v,w)=0$
for every $w\in T_{v}L$, which, together with \eqref{eq1-lemma1-RandersFoliation-minkowskiSpace},      implies     that
$\langle w, \dualbeta \rangle =0,$
for every $w\in T_{\gamma(t_{0})}L$.

\end{proof}

%%%%%%%

\begin{lemma}
\label{lemma6-RandersFoliation-minkowskiSpace}
The vector $\dualbeta$ (and hence the wind) is tangent to the minimal stratum $\Sigma$. 
\end{lemma}
\begin{proof}
  Recall that the tangent space to the future sphere $S^+(0)=\{v\in \RR^n: Z(v)=1\}$ is 
	given by the vectors $w\in \RR^n$ such that $\metric_{v}(v,w)=0$ as $w(Z)=2g_v(v,w)$. 
	Let $\delta>0$ be such that $p=\delta \dualbeta\in S^+(0)$. From the last observation and \eqref{eq1-lemma1-RandersFoliation-minkowskiSpace}, it follows that $\dualbeta$ is orthogonal (with the Euclidean metric) to the tangent space to $S^+(0)$ in $p$.
	Alternatively this fact can  be directly checked through calculations using Lemma \ref{lemma-Randes-equivalence}.
	Since $\F$ is a singular Finsler foliation, the leaf $L_{p}$ is contained in  $S^{+}(0)$.
 On the other hand, it follows from  Lemma \ref{lemma1-RandersFoliation-minkowskiSpace} that $L_{p}$ is contained in $V=\{y \in \mathbb{R}^{n}, \langle y, \dualbeta\rangle=c\}$ for some $c\in\RR$, which is tangent to $S^+(0)$ as $\dualbeta$ is orthogonal to $S^+(0)$. By the strong convexity of $S^+(0)$, $V\cap S^+(0)=\{p\}$ and therefore $L_{p}$ must be just the point $p$.
In other words, we conclude that $0$ and $p=\delta \dualbeta$ (for some $\delta>0$) are contained in the minimal stratum $\Sigma$ and Lemma \ref{lemma4-RandersFoliation-minkowskiSpace} concludes.

\end{proof}

\begin{remark}
\label{remark-F-SRF-randersminkowiski}
       Let $\Sigma^{\perp}$ be the orthogonal space to $\Sigma$ with respect to the Euclidean metric  $\langle \cdot, \cdot \rangle$  at some point $p\in \Sigma$.  Then the above lemma implies that 
the vector $\dualbeta$ is orthogonal  to $\Sigma^{\perp}$ and hence $\F \cap \Sigma^{\perp}$ is an S.R.F.
Using the homothetic transformation lemma, it is possible to prove  that \emph{$\F$ is a singular Riemannian foliation 
on $(\mathbb{R}^{n}, \langle \cdot, \cdot \rangle)$}.      
\end{remark}

%%%%%%%%%%%%%%%%%%%%%%%%%%%%%%%%%%%%%%%%%%%%%%%%%%%%%%%%%%%%%%%%%%%%%%%%%%%%%%%%%%%%%%%%%%%%%%%%%%%%%%
%%%%%%%%%%%%%%%%%%%%%%%%%%%%%%%%%%%%%%%%%%%%% proof SINGULAR FINSLER FOLATION ON RANDERS SPACE
%%%%%%%%%%%%%%%%%%%%%%%%%%%%%%%%%%%%%%%%%%%%%%%%%%%%%%%%%%%%%%%%%%%%%%%%%%%%%%%%%%%%%%%%%%%%%%%%%%%%%%

\subsection{Proof of Theorem \ref{theorem-Randersfoliation}}

The strategy of the proof is simple. In order to prove that a singular Finsler foliation $\F$ on a Randers space $M$  with Zermelo data 
$(\metriczermelo,W)$ 
is a singular Riemannian foliation  with respect to the metric $\metriczermelo$,
we start by showing that the wind $W$ is tangent to each singular stratum, see Proposition \ref{proposition-wind-tangent}. This is done  by  applying  the \emph{slice  foliation};     recall
Proposition \ref{proposition-sliceflat-metric}.    

Once we have shown that the wind is tangent to each  stratum $\Sigma$, we have that $\F|_{\Sigma}$, i.e., the restricted foliation $\F$ to the stratum $\Sigma$,
is a (regular) Finsler foliation on the Randers space $\Sigma$ with Zermelo data $(\metriczermelo,W)$. From Proposition \ref{proposition-Randes-strata} we will conclude that $\F|_{\Sigma}$ is a Riemannian foliation with respect to $\metriczermelo$. 

Finally we will    apply \cite[Proposition 2.15]{Alex6} to conclude that $\F$ is an S.R.F. on $(M,\metriczermelo)$, because $\F|_{\Sigma}$ is a Riemannian foliation on  each stratum $(\Sigma,\metriczermelo)$.  Alternatively one can use  Proposition \ref{proposition-foliation-submersion}.

%%%%%%%%%%%%%

\begin{proposition}
\label{proposition-wind-tangent}
Let $\F=\{L\}$ be  a singular Finsler foliation with locally closed leaves on a  Randers space $(M,Z)$
 with Zermelo data $(\metriczermelo, W)$. 
Then the wind $W$ is always tangent to the strata of $\F$.
\end{proposition}
\begin{proof}
Let $(S_q,\hat F)$ be the singular Finsler foliation obtained in Proposition \ref{proposition-metricFS} for $q\in M$.
Let $d\pi_q:T_{q}M\to T_qS_{q}$ be the map defined in \eqref{lin-sub}. Observe that $\hat F_q$ is a Zermelo metric with wind $d\pi_q(W)$ (see Proposition \ref{proposition-translation-of-submersion}).
 Recall that the fiber $(d\pi)^{-1}_{q}(0)$ coincides with $T_{q}L_{q}$, which is 
contained in $T_{q}\Sigma$, i.e., in the tangent space of the stratum containing  the point $q$. Note that $W_{q}=d\pi_{q}(W)+W^{\perp}_{q}$, where
$W^{\perp}_{q}\in T_{q}L_{q}\subset T_{q}\Sigma$. Therefore, in order to prove that $W_{q}\in T_{q}\Sigma$, it suffices to check that
$d\pi_{q}(W_{q})\in T_{q}\Sigma$. But,  on the other hand, this follows from 
Lemma \ref{lemma6-RandersFoliation-minkowskiSpace} and Proposition \ref{proposition-sliceflat-metric}.

\end{proof}

%%%
\begin{proposition}
\label{proposition-Randes-strata}
Let $\F=\{L\}$ be a (regular) Finsler foliation on a  Randers space $(\Sigma,Z)$ with Zermelo data $(\metriczermelo,W)$. 
 Then $\F$  is a (regular) Riemannian foliation with respect to $\metriczermelo$. 
\end{proposition}
\begin{proof}
  As a regular Finsler foliation is a Finsler submersion locally (see Corollary~ \ref{equivalencedef}), the result follows from Proposition \ref{proposition-translation-of-submersion}. 

\end{proof}

\begin{remark}
Consider a singular Finsler foliation $\F$        with locally closed leaves      on  a Randers manifold $(M,Z)$ with Zermelo data $(\metriczermelo,W)$. 
Then  Propositions~\ref{proposition-wind-tangent}, \ref{proposition-stratification}
and \ref{proposition-translation-of-submersion}
 imply that  the  wind $W$ is an $\F$-foliated vector field. 
\end{remark}

%%%%%%%%%%%%%%%%%%%%%%%%%%%%%
%%%%%%%%%%%%%%%%% corollary
%%%%%%%%%%%%%%%%%%%%%

\subsection{Proof of Corollary \ref{the-corollary}}

In this section we assume that  the wind  $W$ of a  Randers space 
$(M,Z)$ with Zermelo data $(\metriczermelo,W)$ is an infinitesimal homothety of $\metriczermelo$, i.e., 
$\mathcal{L}_{W} \metriczermelo=-\sigma \metriczermelo$.
We also assume that $\metriczermelo$ and the wind $W$ are complete. This implies that the metric $Z$ is also complete;
recall \cite[Theorem 1.2]{Javaloyes-Vitorio-Zermelo}  and \cite{Huang-Mo-geodesics}. 

Since $\F$ is a singular Finsler foliation  on $(M,Z)$, 
Theorem \ref{theorem-Randersfoliation} implies that $\F$  is a singular Riemannian
foliation on $(M,\metriczermelo)$. Therefore, it follows from \cite{AlexToeben2} that each regular leaf of $\F$ is equifocal with 
respect to $\metriczermelo$. In other words,
we have the following result:
\begin{lemma}[\cite{AlexToeben2}]
\label{lemma-SRF-equifocal} 
For each $p$ in a regular leaf $L$, there exists a neighborhood $U\subset L$ of $p$ such that for each (unit) basic vector 
$\tilde{\xi}$  on $U$ (with respect to $\metriczermelo$)  the endpoint map $\eta_{s \tilde{\xi}}:U\to M$, 
defined as $\eta_{s\tilde{\xi}}(x)=\exp_{x}(s\tilde{\xi})$, with $s>0$ 
fulfills the following properties: 
\begin{itemize}
\item the derivative of $\eta_{s\tilde{\xi}}$ has constant rank,
\item $\eta_{s\tilde{\xi}}(U)$ is an open set of $L_{q}$, where $q=\eta_{s\tilde{\xi}}(p)$.
\end{itemize}
\end{lemma}
Our goal in this section is to prove that $L$ is equifocal with respect to $Z$.
 More precisely, we are going to check  that the above properties hold 
for each $t\xi$ where $\xi$ is a  (unit) normal vector field (with respect to $Z$) on the neighborhood $U$ defined above and  $t>0$. 

In order to prove this, we need to recall the next   result.

	\begin{lemma}[\cite{Javaloyes-Vitorio-Zermelo,Huang-Mo-geodesics,Robles-geodesics-constant-curvature}]
	\label{c.e.geodesicas}
	A curve  $\gamma:\mathbb{R}\to M$ is a geodesic with unit velocity on  $(M,Z)$ if and only if 
		$\gamma(t)=\varphi_t(\tilde{\gamma}(t))$ where  $\tilde{\gamma}:\mathbb{R}\to M$
		is a reparametrization of geodesic with respect to $\metriczermelo$  so that 
		$\metriczermelo(\tilde{\gamma}'(t),\tilde{\gamma}'(t))=e^{-\sigma t}$, where  $\varphi$ is the flow of  $W$. In particular, by deriving  at $t=0$, 
		we have that $\gamma'(0)=\tilde{\gamma}'(0)+W(\gamma(0))$. 
	\end{lemma}
	
	From Lemma  \ref{lemma-product-gvvu}, we can also conclude that: 
	
	\begin{lemma}
	\label{orthogonal}
	If $\xi$ is a unit vector field on $U$ orthogonal to $L$  
	with respect to $Z$ then $\xi-W$ is orthogonal to $L$ (with respect to $\metriczermelo$). 
\end{lemma}

We are now ready to prove the corollary. Let $\xi$ be  a unit basic vector field on $U$ with respect to $Z$. 
For $x\in U$ set $t\to \gamma_{x}(t)=\eta_{t\xi}(x)$. Let $\tilde{\gamma}_{x}$ be the reparametrization of geodesic defined in
Lemma \ref{c.e.geodesicas}, i.e., $\gamma_{x}(t)=\varphi_{t}(\tilde{\gamma}_{x}(t))$.
Set $\tilde{\xi}:=\xi-W$. Note that $\tilde{\gamma}_{x}'(0)=\tilde{\xi}(x)$ (recall Lemma \ref{c.e.geodesicas}).
Due to Lemma \ref{orthogonal}, the vector field $\tilde{\xi}$ is orthogonal to $L$. Also note that $\tilde{\xi}$ is projectable
and hence basic, because $W$ is a foliated vector field (recall Theorem \ref{theorem-Randersfoliation}) and  that $\metriczermelo(\tilde{\xi},\tilde{\xi})=1$. 
We can now infer that 
\begin{equation}
\label{eq-endpointmaps-Z-h}
\eta_{t\xi}(x)=\gamma_{x}(t)=\varphi_{t}(\tilde{\gamma}_{x}(t))=\varphi_{t}\circ\eta_{s\tilde{\xi}}(x)
\end{equation}
where $s=\frac{-2}{\sigma}\Big( \exp(-\frac{\sigma}{2}t)-1 \Big)$, if $\sigma\neq 0$ and $s=t$, if $\sigma=0$. Since $t$ and $s$ are fixed and $W$ is foliated (i.e., the flow $\varphi_{t}$ sends leaves to leaves) we conclude  from the above
equation and Lemma \ref{lemma-SRF-equifocal} that $L$ is equifocal, i.e.,  
 the derivative of $\eta_{t \xi}$ has constant rank and $\eta_{t \xi}(U)$ is an open set of $L_{q}$, where $q=\eta_{t \xi}(p)$.

%%%%%%%%%%%%%%%%%%%%%%%%%%%%%%%
%%%%%%%%%%%%%%%%%%%%%%%%
%%%%%%%%%%%%%%%%%%%%%%%%%%%%%
%%%%%%%%%%%%%%%%%%%%%%%%%%%%%%%%%%%%%%%%%%%%%%%%

\bibliographystyle{amsplain}

\end{document}